\documentclass[12pt,a4paper,nosumlimits]{amsart}
\usepackage{amssymb,amsxtra,eucal}
\usepackage[cmtip,arrow]{xy}
\usepackage{pb-diagram,pb-xy}

\calclayout

\newcommand{\+}{\nobreakdash-}
\renewcommand{\:}{\colon}
\renewcommand{\;}{,\medspace}
\renewcommand{\.}{\text{$\mskip .5\thinmuskip$}}

\renewcommand{\le}{\leqslant}
\renewcommand{\ge}{\geqslant}

\newcommand{\rarrow}{\longrightarrow}

\newcommand{\birarrow}{\rightrightarrows}
\newcommand{\ot}{\otimes}
\newcommand{\ocn}{\odot}
\newcommand{\Ocn}{\circledcirc}
\newcommand{\st}{\star}

\renewcommand{\ss}{{\mathrm{ss}}}
\newcommand{\rop}{{\mathrm{op}}}

\newcommand{\sop}{{\mathsf{op}}}

\newcommand{\oc}{\mathbin{\text{\smaller$\square$}}}
\newcommand{\bu}{{\text{\smaller\smaller$\scriptstyle\bullet$}}}
\newcommand{\lrarrow}{\.\relbar\joinrel\relbar\joinrel\rightarrow\.}

\newcommand{\B}{{\mathcal B}}
\newcommand{\C}{{\mathcal C}}
\newcommand{\D}{{\mathcal D}}
\newcommand{\E}{{\mathcal E}}
\newcommand{\K}{{\mathcal K}}
\renewcommand{\L}{{\mathcal L}}
\newcommand{\M}{{\mathcal M}}
\newcommand{\N}{{\mathcal N}}
\newcommand{\I}{{\mathcal I}}
\newcommand{\J}{{\mathcal J}}

\newcommand{\fR}{{\mathfrak R}}
\newcommand{\fF}{{\mathfrak F}}
\newcommand{\fK}{{\mathfrak K}}
\renewcommand{\P}{{\mathfrak P}}
\newcommand{\Q}{{\mathfrak Q}}

\renewcommand{\S}{{\boldsymbol{\mathcal S}}}
\newcommand{\T}{{\boldsymbol{\mathcal T}}}
\newcommand{\bB}{{\boldsymbol{\mathcal B}}}
\newcommand{\bK}{{\boldsymbol{\mathcal K}}}
\newcommand{\bL}{{\boldsymbol{\mathcal L}}}
\newcommand{\bM}{{\boldsymbol{\mathcal M}}}
\newcommand{\bN}{{\boldsymbol{\mathcal N}}}
\newcommand{\bF}{{\boldsymbol{\mathfrak F}}}
\newcommand{\bP}{{\boldsymbol{\mathfrak P}}}
\newcommand{\bQ}{{\boldsymbol{\mathfrak Q}}}

\renewcommand{\b}{{\mathsf{b}}}
\newcommand{\co}{{\mathsf{co}}}
\newcommand{\ctr}{{\mathsf{ctr}}}
\newcommand{\si}{{\mathsf{si}}}
\newcommand{\abs}{{\mathsf{abs}}}
\newcommand{\inj}{{\mathsf{inj}}}
\newcommand{\proj}{{\mathsf{proj}}}

\newcommand{\boR}{{\mathbb R}}
\newcommand{\boL}{{\mathbb L}}
\newcommand{\Z}{{\mathbb Z}}

\newcommand{\sD}{{\mathsf D}}
\newcommand{\Hot}{{\mathsf{Hot}}}

\newcommand{\vect}{{\operatorname{\mathsf{--vect}}}}
\newcommand{\modl}{{\operatorname{\mathsf{--mod}}}}
\newcommand{\comodl}{{\operatorname{\mathsf{--comod}}}}
\newcommand{\comodr}{{\operatorname{\mathsf{comod--}}}}
\newcommand{\bicomod}{{\operatorname{\mathsf{--comod--}}}}
\newcommand{\simodl}{{\operatorname{\mathsf{--simod}}}}
\newcommand{\simodr}{{\operatorname{\mathsf{simod--}}}}
\newcommand{\sicntr}{{\operatorname{\mathsf{--sicntr}}}}

\newcommand{\smooth}{{\operatorname{\mathsf{--smooth}}}}
\newcommand{\contra}{{\operatorname{\mathsf{--contra}}}}

\newcommand{\tors}{{\operatorname{\mathsf{-tors}}}}
\newcommand{\ctra}{{\operatorname{\mathsf{-ctra}}}}

\DeclareMathOperator{\Hom}{Hom}
\DeclareMathOperator{\Cohom}{Cohom}

\DeclareMathOperator{\Tor}{Tor}
\DeclareMathOperator{\Ctrtor}{Ctrtor}
\DeclareMathOperator{\CtrTor}{CtrTor}

\newcommand{\Section}[1]{\bigskip\section{#1}\medskip}
\theoremstyle{plain}
\newtheorem{thm}{Theorem}[section]

\newtheorem{lem}[thm]{Lemma}
\newtheorem{prop}[thm]{Proposition}

\theoremstyle{definition}
\newtheorem{ex}[thm]{Example}
\newtheorem{exs}[thm]{Examples}

\begin{document}

\title{Dedualizing complexes of bicomodules \\
and MGM duality over coalgebras}

\author{Leonid Positselski}

\address{Department of Mathematics, Faculty of Natural Sciences,
University of Haifa, Mount Carmel, Haifa 31905, Israel; and
\newline\indent Laboratory of Algebraic Geometry, National Research
University Higher School of Economics, Moscow 119048; and
\newline\indent Sector of Algebra and Number Theory, Institute for
Information Transmission Problems, Moscow 127051, Russia; and
\newline\indent Charles University in Prague, Faculty of Mathematics
and Physics, Department of Algebra, Sokolovsk\'a~83,
186~75 Praha, Czech Republic}

\email{posic@mccme.ru}

\begin{abstract}
 We present the definition of a dedualizing complex of bicomodules over
a pair of cocoherent coassociative coalgebras $\C$ and~$\D$.
 Given such a complex $\B^\bu$, we construct an equivalence between
the (bounded or unbounded) conventional, as well as absolute, derived
categories of the abelian categories of left comodules over $\C$ and
left contramodules over~$\D$.
 Furthermore, we spell out the definition of a dedualizing complex of
bisemimodules over a pair of semialgebras, and construct the related
equivalence between the conventional or absolute derived categories
of the abelian categories of semimodules and semicontramodules.
 Artinian, co-Noetherian, and cocoherent coalgebras are discussed as
a preliminary material.
\end{abstract}

\maketitle

\tableofcontents

\section{Introduction}
\medskip

\subsection{{}}
 In the classical homological algebra, one was not supposed to consider
unbounded derived categories; certainly not when working with
categories or functors of infinite homological dimension.
 Right derived functors were acting from bounded below derived
categories, while left derived functors were defined on bounded above
derived categories.
 Such derived functors were constructed using resolutions by complexes
of injective or projective objects.
 More generally, one would consider resolutions by complexes with
the terms adjusted to the particular functor in question, such as
flat modules or flasque sheaves.

 Everything changed after the watershed paper of
Spaltenstein~\cite{Spa}, which explained, following the idea of
Bernstein, how to work with unbounded complexes, particularly when
constructing derived functors of infinite homological dimension.
 The key innovation was to strengthen the conditions imposed on
resolutions and work with what came to be known as
the \emph{homotopy injective} or \emph{homotopy projective} complexes.
 Unlike in the classical homological algebra, these are not termwise
conditions: whether a complex is homotopy injective or homotopy
projective depends on the differential in the complex and not only on
its terms.

 The unbounded derived category of modules over an associative ring
turned out to be particularly well-behaved.
 In the subsequent work of Keller, Bernstein--Lunts, and
Hinich~\cite{Kel,BL,Hin}, the theory was extended to DG\+modules over
DG\+rings.
 The unbounded derived category of DG\+modules $\sD(A\modl)$ is
compactly generated, and it only depends on the quasi-isomorphism
class of a DG\+ring~$A$.
 One can generalize even further and replace an associative DG\+ring
$A$ with an $\mathrm{A}_\infty$\+algebra.

 The theory that grew out of Spaltenstein's paper became so hugely
popular that nowadays people use the homotopy injective and homotopy
projective resolutions even when they are not actually relevant.
 That is what was happening in the case of the MGM
(Matlis--Greenlees--May) duality/equivalence theory~\cite{PSY}.
 In fact, in the MGM theory one deals with derived functors of finite
homological dimension, and the use of complexes of adjusted objects,
similar to that of (complexes of) flasque or soft sheaves in
the computation of sheaf cohomology/derived direct images, is called
for~\cite{Pmgm}.

\subsection{{}}
 The next development, which came about a decade later, was that people
started to work with complexes viewed up to equivalence relations
more delicate than the conventional
quasi-isomorphism~\cite{Hin2,Lef,Kel2}.
 In other words, triangulated categories in which some, though not too
many, acyclic complexes survive as nonzero objects attracted
a certain interest.
 In addition to the constructions of compact generators~\cite{Jor,Kra},
one of manifestations of the phenomenon which the present author calls
the \emph{derived co-contra correspondence} was first noticed
in the paper~\cite{IK}.

 The present author's own ideas about the subject were published with
about a decade-long delay~\cite{Psemi,Pkoszul}.
 Developed originally in the context of derived nonhomogeneous Koszul
duality, they proceed from the observation that replacing
the conventional quasi-isomorphism of complexes with more delicate
equivalence relations is a natural alternative to strengthening
the conditions on resolutions when working with unbounded complexes.
 In the terminology going back to the classical paper~\cite{HMS}
(where \emph{two kinds of differential derived functors} were
introduced), this point of view came to be known as the distinction
between \emph{two kinds of derived categories}.

 In the derived categories of the first kind, complexes are considered
up to the conventional quasi-isomorphism (which does not depend on
the module structure on the complexes, but only on their underlying
complexes of abelian groups), which necessitates the use of homotopy
adjusted complexes as resolutions (meaning the conditions on resolving
complexes depending on the differentials and not only on
the underlying graded object structures).
 In the derived categories of the second kind, some acyclic complexes
survive as nonzero objects (and the equivalence relation on complexes
depends on their module structures and not only on the underlying
complexes of abelian groups), while the conditions on resolutions
do not depend on the differentials on them (but only on their
underlying graded objects).

 Another advantage of derived categories of the second kind is that
they are defined for \emph{curved} differential graded (CDG) structures
as well as for conventional differential graded structures~\cite{Pkoszul}.
 Hence the important role that such derived category constructions play,
in particular, in the theory of matrix factorizations~\cite{Or,EP,BDFIK}.

 The conventional derived category is the derived category of the first
kind.
 The two most important versions of derived categories of the second
kind are the \emph{coderived} and the \emph{contraderived} category.
 In well-behaved situations, the coderived category of (curved)
DG\+modules is equivalent to the homotopy category of (curved)
DG\+modules whose underlying graded modules are injective, while
the contraderived category of (curved) DG\+modules is equivalent
to the homotopy category of (curved) DG\+modules whose underlying
graded modules are projective.

\subsection{{}}
 It turned out that the conventional derived categories of
DG\+comodules over DG\+coalgebras (over a field) are not as
well-behaved as the derived categories of DG\+modules.
 The derived category of DG\+comodules over a DG\+coalgebra can change
when the DG\+coalgebra is replaced by a quasi-isomorphic
one~\cite{Kal}, \cite[Remark~2.4]{Pkoszul}.
 There are no obvious reasons why the derived category $\sD(\C\comodl)$
of DG\+comodules over a DG\+coalgebra $\C$ (or even complexes of
comodules over a coalgebra~$\C$) should be compactly generated, though
one can show that it is well-generated~\cite[Section~5.5]{Pkoszul}.

 Perhaps one is not supposed to consider the conventional derived
category of complexes of comodules (or DG\+comodules) at all.
 The notion that one should work with the derived categories of
modules and the coderived categories of comodules goes back
to~\cite{Lef,Kel2}.
 The monograph~\cite{Psemi} is based on the philosophy that one is
supposed to take the derived category of modules, the coderived
category of comodules, and the contraderived category of contramodules.

 In fact, for any curved DG\+coalgebra $\C$ over a field~$k$,
the coderived category $\sD^\co(\C\comodl)$ of left curved
DG\+comodules over $\C$ is compactly generated (by the bounded
derived category of $k$\+finite-dimensional CDG\+comodules).
 The coderived category of CDG\+comodules is also equivalent to
the homotopy category of CDG\+comodules with injective underlying
graded comodules.

 Furthermore, there is a natural equivalence between
the coderived category of left CDG\+comodules and
the contraderived category of left CDG\+contramodules
over~$\C$\, \cite[Section~5]{Pkoszul}:
\begin{equation} \label{co-contra-correspondence}
 \sD^\co(\C\comodl)\simeq\sD^\ctr(\C\contra).
\end{equation}
 The contraderived category of CDG\+contramodules is equivalent to
the homotopy category of CDG\+contramodules with projective underlying
graded contramodules.
 The triangulated equivalence~\eqref{co-contra-correspondence} is
a principal example of what is called the \emph{derived
comodule-contramodule correspondence} phenomenon
in~\cite{Psemi,Pkoszul}.

\subsection{{}}
 Several words about the \emph{contramodules} are due at this point.
 There are two abelian categories associated naturally with
an associative ring $A$: the left $A$\+modules and the right
$A$\+modules.
 In contrast, for a coalgebra $\C$ (say, over a field~$k$) there are
\emph{four} such abelian categories: the left and the right
\emph{$\C$\+comodules}, and the left and the right
\emph{$\C$\+contramodules}.
 The categories of comodules have exact functors of filtered inductive
limit and enough injective objects.
 The categories of contramodules have exact functors of infinite
product and enough projective objects.

 Let $\C^*$ denote the dual $k$\+vector space to $\C$, endowed with
its natural structure of a pro-finite-dimensional topological
$k$\+algebra.
 Then the $\C$\+comodules are the same thing as \emph{discrete}
$\C^*$\+modules, while the $\C$\+contramodules form an ``intermediate''
category between arbitrary $\C^*$\+modules and \emph{pseudo-compact}
(pro-finite-dimensional) topological $\C^*$\+modules.
 The latter form a category equivalent to the opposite category to
$\C$\+comodules.
 More precisely, there are natural forgetful functors
$$
 (\comodr\C)^\sop\lrarrow\C\contra\lrarrow\C^*\modl
$$
from the opposite category to right $\C$\+comodules to left
$\C$\+contramodules and to left $\C^*$\+modules (in addition
to the fully faithful functor $\C\comodl\rarrow\C^*\modl$ identifying
left $\C$\+comodules with discrete left $\C^*$\+modules).

 In other words, $\C$\+contramodules can be viewed as a species of
``complete'' (as opposed to discrete) $\C^*$\+modules.
 Nevertheless, contramodules carry no underlying topologies on them.
 Instead, they are discrete $k$\+vector spaces endowed with
\emph{infinite summation operations} with the coefficients in~$\C^*$
\cite{Prev}.

\subsection{{}}
 As we have already mentioned, the coderived categories of comodules
and the contraderived categories of contramodules are better behaved
than the conventional (unbounded) derived categories of comodules
or contramodules.
 In other words, considering derived categories of the first kind
along the ring variables and derived categories of the second kind
along the coalgebra variables produces the better behaved
triangulated categories~\cite{Psemi}.

 Still, there is something to be said about the conventional
derived categories of DG\+comodules and DG\+contramodules, too.
 The following results can be found in~\cite[Theorem~2.4 and
Section~5.5]{Pkoszul} (for part~(d), one has to look into
the postpublication \texttt{arXiv} version of~\cite{Pkoszul}).

\begin{thm}
 Let\/ $\C$ be a DG\+coalgebra over a field~$k$.  Then \par
\textup{(a)} the Verdier quotient functor\/ $\sD^\ctr(\C\contra)
\rarrow\sD(\C\contra)$ has a (fully faithful) left adjoint functor\/
$\sD(\C\contra)\rarrow\sD^\ctr(\C\contra)$; \par
\textup{(b)} the essential image of the triangulated functor\/
$\sD(\C\contra)\rarrow\sD^\ctr(\C\contra)$ is the minimal full
triangulated subcategory in\/ $\sD^\ctr(\C\contra)$ containing
the left DG\+contramodule\/ $\Hom_k(\C,k)$ over\/ $\C$ and closed
under infinite direct sums; \par
\textup{(c)} the Verdier quotient functor\/ $\sD^\co(\C\comodl)
\rarrow\sD(\C\comodl)$ has a (fully faithful) right adjoint functor\/
$\sD(\C\comodl)\rarrow\sD^\co(\C\comodl)$; \par
\textup{(d)} assuming Vop\v enka's principle in set theory,
the essential image of the triangulated functor\/
$\sD(\C\comodl)\rarrow\sD^\co(\C\comodl)$ is the minimal full
triangulated subcategory in\/ $\sD^\co(\C\comodl)$ containing
the left DG\+comodule\/ $\C$ over\/ $\C$ and closed under
infinite products.
\end{thm}

 It should be added that the triangulated
equivalence~\eqref{co-contra-correspondence} takes the DG\+comodule
$\C$ over $\C$ to the DG\+contramodule $\Hom_k(\C,k)$ over~$\C$.
 Thus, assuming Vop\v enka's principle, the derived categories
$\sD(\C\comodl)$ and $\sD(\C\contra)$ are related as two full
triangulated subcategories in the same triangulated
category~\eqref{co-contra-correspondence}, generated by the same object
in this category; but one of them is generated using shifts, cones,
and infinite products, while the other one is generated using shifts,
cones, and infinite direct sums.
 Both the (set-indexed) direct sums and products of arbitrary objects
exist in the compactly generated triangulated
category~\eqref{co-contra-correspondence}\, (but the object
$\C\longleftrightarrow\Hom_k(\C,k)$ is not compact in
this category). 

\subsection{{}}
 The aim of this paper is to demonstrate a set of (admittedly, rather
restrictive) assumptions and additional data allowing to construct
an equivalence between the derived category of complexes of comodules
over a coalgebra $\C$ and the derived category of complexes of
contramodules over another coalgebra~$\D$,
\begin{equation} \label{mgm-equivalence}
 \sD^\st(\C\comodl)\simeq\sD^\st(\D\contra).
\end{equation}
 Here the symbol $\st$ means that both the bounded and unbounded
conventional derived categories are allowed, i.~e., one can have
$\st=\b$, $+$, $-$, or~$\varnothing$.

 Moreover, the triangulated equivalence~\eqref{mgm-equivalence}
also holds for the \emph{absolute derived categories} with the symbols
$\st=\abs+$, $\abs-$, or~$\abs$, which are versions of the construction
of derived categories of the second kind introducted in~\cite{Pkoszul}
and~\cite[Appendix~A]{Pcosh}.
 Unlike in~\eqref{co-contra-correspondence}, though, the derived
category symbol must be the same in the left and the
right-hand side of the equivalence~\eqref{mgm-equivalence}.

 The triangulated equivalence~\eqref{mgm-equivalence}, connecting
the conventional derived categories of comodules and contramodules,
is a species of what can be called the ``na\"\i ve derived co-contra
correspondence''.
 In the present author's work, it first appeared in
the algebro-geometric setting as an equivalence between the derived
categories of quasi-coherent sheaves and contraherent
cosheaves over a quasi-compact semi-separated
scheme~\cite[Section~4.6]{Pcosh} (or alternatively, over
a Noetherian scheme of finite Krull
dimension~\cite[Theorem~5.8.1]{Pcosh}).
 In the subsequent papers~\cite{Pmgm,PMat}, the same principle was
applied in order to formulate the \emph{MGM duality/equivalence}
and the \emph{triangulated Matlis equivalence}.

\subsection{{}}
 The equivalence of categories~\eqref{mgm-equivalence} can be called
the ``MGM duality for coalgebras''.
 A bit of history of the MGM duality is worth recalling in this
connection.
 The three-letter abbreviation stands for
\emph{Matlis--Greenlees--May}~\cite{Mat2,GM}.
 The related chain of results can be further traced to the seminal
paper of Harrison~\cite{Har}, where certain equivalences of additive
subcategories in the category of abelian groups were constructed.
 Matlis extended these to equivalences between additive subcategories
in the category of modules over an arbitrary commutative
domain~\cite{Mat1}.

 In the paper~\cite{Mat2}, which came more than a decade later,
Matlis constructs an equivalence between certain additive subcategories
in the category of modules over a commutative ring $R$ related to
an ideal $I\subset R$ generated by a regular sequence.
 Greenlees and May~\cite{GM} initiated the study of the derived
functors of $I$\+adic completion for an arbitrary finitely generated
ideal $I$ in a commutative ring~$R$.
 Dwyer and Greenlees~\cite{DG} formulated the theory in the form of
a triangulated equivalence between two full subcategories (of what
we would now call the ``$I$\+torsion'' and ``$I$\+complete'' complexes)
in the derived category $\sD(R\modl)$ of modules over a commutative
ring $R$ with a finitely generated ideal $I\subset R$.
 Porta, Shaul, and Yekutieli~\cite{PSY} studied the case of
a \emph{weakly proregular} finitely generated ideal~$I$.

 The present author's paper~\cite{Pmgm}, which formulated the theory
in its state-of-the-art form, emphasized and discussed the role of
what it called a \emph{dedualizing complex} of $I$\+torsion
$R$\+modules in the MGM duality theory.
 It also demonstrated the central role of the abelian category of
\emph{$I$\+contramodule $R$\+modules}, on par with the much more
familiar dual-analogous abelian category of \emph{$I$\+torsion
$R$\+modules}, in the MGM duality.

\subsection{{}}
 The main results of the MGM duality theory, as formulated
in~\cite{Pmgm}, are the following ones.
 Given a finitely generated ideal $I$ in a commutative ring $R$,
denote by $R\modl_{I\tors}$ and $R\modl_{I\ctra}\subset R\modl$
the abelian subcategories of $I$\+torsion and $I$\+contramodule
$R$\+modules.
 (See~\cite{Pcta} for an introductory discussion of these
subcategories.)
 Then for every conventional derived category symbol $\st=\b$, $+$,
$-$, or~$\varnothing$ there is a natural triangulated equivalence
\begin{equation}
 \sD^\st_{I\tors}(R\modl)\simeq\sD^\st_{I\ctra}(R\modl)
\end{equation}
between the full subcategory of complexes of $R$\+modules with
$I$\+torsion cohomology modules and the full subcategory of
complexes of $R$\+modules with $I$\+contramodule cohomology
modules in $\sD^\st(R\modl)$.

 Furthermore, assuming that the ideal $I\subset R$ is weakly proregular
(which always holds, e.~g., when the ring $R$ is Noetherian),
for every derived category symbol $\st=\b$, $+$, $-$, $\varnothing$,
$\abs+$, $\abs-$, or~$\abs$, there is a natural equivalence between
the derived categories of the abelian categories
$R\modl_{I\tors}$ and $R\modl_{I\ctra}$,
\begin{equation} \label{commutative-algebra-mgm-equivalence}
 \sD^\st(R\modl_{I\tors})\simeq\sD^\st(R\modl_{I\ctra}).
\end{equation}
 The triangulated equivalence~\eqref{mgm-equivalence} is
a noncocommutative coalgebra version of the triangulated
equivalence~\eqref{commutative-algebra-mgm-equivalence}.

 When $R$ is a finitely generated algebra over an algebraically
closed field~$k$ and $I$ is a maximal ideal in~$R$,
the equivalence~\eqref{commutative-algebra-mgm-equivalence} becomes
a particular case of the equivalence~\eqref{mgm-equivalence}.
 (See the discussion in~\cite[Section~0.10]{Pmgm} and generally in
the introduction to~\cite{Pmgm}, where the conceptual importance of
coalgebra-related considerations in the MGM duality theory is also
emphasized.)

\subsection{{}}
 The triangulated equivalence~\eqref{mgm-equivalence} depends on
an additional piece of data called a \emph{dedualizing complex of\/
$\C$\+$\D$\+bicomodules}~$\B^\bu$.
 The definition of a dedualizing complex of bicomodules is dual to
that of a dualizing complex of bimodules for a pair of associative
rings~\cite{Yek,YZ,Miy,CFH,Pfp}.

 A detailed discussion of the related philosophy can be found
in the introduction to~\cite{Pmgm}.
 To recall it very briefly here, let us mention that an associative
ring $A$ is itself a dedualizing complex of $A$\+$A$\+bimodules.
 Given a \emph{dualizing} complex of $A$\+$B$\+bimodules $D^\bu$ for
a pair of associative rings $A$ and $B$, one constructs a triangulated
equivalence between the coderived and the contraderived category
of modules
\begin{equation} \label{covariant-serre-grothendieck}
 \sD^\co(A\modl)\simeq\sD^\ctr(B\modl),
\end{equation}
which can be called the \emph{covariant Serre--Grothendieck
duality}~\cite{Pfp}.

 Conversely, a coalgebra $\C$ over a field~$k$ is itself
a \emph{dualizing} complex of $\C$\+$\C$\+bico\-modules; hence
the triangulated equivalence~\eqref{co-contra-correspondence}.
 The datum of a \emph{dedualizing} complex of $\C$\+$\D$\+bicomodules
allows to construct a triangulated equivalence~\eqref{mgm-equivalence}.

\subsection{{}}
 Furthermore, a (semiassociative and semiunital) \emph{semialgebra}
$\S$ over a coalgebra $\C$ over a field~$k$ is an algebra object in
the (noncommutative, but associative and unital) tensor category of
bicomodules over $\C$ with respect to the operation of \emph{cotensor
product}~$\oc_\C$.
 Just as for a coalgebra $\C$, there are \emph{four} module categories
naturally assigned to a semialgebra $\S$: the left and right
\emph{$\S$\+semimodules}, and the left and right
\emph{$\S$\+semicontramodules}.
 The category of left $\S$\+semimodules $\S\simodl$ is abelian and
the forgetful functor $\S\simodl\rarrow\C\comodl$ is exact if and
only if $\S$ is an injective right $\C$\+comodule.
 The category of left $\S$\+semicontramodules $\S\sicntr$ is abelian
and the forgetful functor $\S\sicntr\rarrow\C\contra$ is exact if and
only if $\S$ is an injective left $\C$\+comodule.

 For any semialgebra $\S$ over a coalgebra $\C$ such that $\S$ is
an injective left $\C$\+comodule and an injective right $\C$\+comodule,
there is a natural equivalence between
the \emph{semiderived categories} of left $\S$\+semimodules and
left $\S$\+semicontramodules~\cite[Sections~0.3.7 and~6.3]{Psemi}:
\begin{equation} \label{semico-semicontra}
 \sD^\si(\S\simodl)\simeq\sD^\si(\S\sicntr).
\end{equation}
 The words ``semiderived category'' actually mean two dual
constructions rather than one: the semiderived category of semimodules 
is what could be more precisely called their \emph{semicoderived}
category, while the semiderived category of semicontramodules could be 
called the \emph{semicontraderived} category.
 These are certain mixtures of the constructions of co- or
contraderived categories (taken ``along~$\C$'') and the conventional 
derived category (taken ``in the direction of $\S$ relative to~$\C$'').

\subsection{{}}
 Now let $\S$ be a semialgebra over a coalgebra $\C$ and $\T$ be
a semialgebra over a coalgebra $\D$, both over the same field~$k$.
 Let $\B^\bu$ be a dedualizing complex of $\C$\+$\D$\+bicomodules.
 In this paper we show that, given a certain further piece of data
called \emph{a dedualizing complex of\/ $\S$\+$\T$\+bisemimodules}
$\bB^\bu$, one can construct a triangulated equivalence between
the conventional derived category of left $\S$\+semimodules and
the conventional derived category of left $\T$\+semicontramodules,
\begin{equation} \label{semialg-conventional-derived}
 \sD(\S\simodl)\simeq\sD(\T\sicntr).
\end{equation}
 Moreover, there are triangulated equivalences
\begin{equation} \label{semialg-mgm-equivalence}
 \sD^\st(\S\simodl)\simeq\sD^\st(\T\sicntr)
\end{equation}
for all the conventional or absolute, bounded or unbounded derived
category symbols $\st=\b$, $+$, $-$, $\varnothing$, $\abs+$,
$\abs-$, or~$\abs$.
 These results can be called the \emph{MGM duality/equivalence for
semialgebras}.

 The definition of a dedualizing complex of bisemimodules is dual to
that of a dualizing complex of bicomodules for a pair of corings over
associative rings~\cite[Section~B.4]{Pcosh}.

\subsection{{}}
 The situation simplifies when the coalgebra $\C$ has finite
homological dimension (i.~e., the abelian category $\C\comodl$ has 
finite homological dimension or, which is equivalent, the abelian 
category $\C\contra$ has finite homological dimension).

 In this case, there is no difference between the semiderived category
$\sD^\si(\S\simodl)$ and the conventional derived category
$\sD(\S\simodl)$, and also no difference between the semiderived
category $\sD^\si(\S\sicntr)$ and the conventional derived category
$\sD(\S\sicntr)$,
$$
 \sD^\si(\S\simodl)=\sD(\S\simodl)\quad\text{and}\quad
 \sD^\si(\S\sicntr)=\sD(\S\sicntr).
$$
 The semialgebra $\S$ itself can be used as a dedualizing complex of
$\S$\+$\S$\+bisemimodules in this case, so~\eqref{semico-semicontra}
becomes an instance of~\eqref{semialg-conventional-derived}
for $\C=\D$ and $\S=\T$.

 In particular, one finds oneself in this situation in the theory of
smooth duality for a $p$\+adic Lie group with coefficients in a field
of characteristic~$p$\,~\cite{Psm}.

\subsection{{}}
 Finally, let us say a few words about the finiteness conditions on
coalgebras, comodules, and contramodules.
 One of the peculiarities of coalgebras is the difference between
the classes of Artinian and co-Noetherian coalgebras or comodules.
 Any Artinian comodule is co-Noetherian, but the converse is not
generally true.
 For a counterexample, one can consider the cosemisimple coalgebra $\C$
that is the direct sum of an infinite number of copies of
the coalgebra~$k$ over~$k$.
 Then $\C$ is a co-Noetherian $\C$\+comodule (i.~e., all its quotient
comodules are finitely cogenerated), but it is not an Artinian object
of the abelian category $\C\comodl$.

 The finiteness conditions on coalgebras were, of course, traditionally
discussed in the language of comodules~\cite{WW,GTNT}.
 Some of the dual-analogous contramodule conditions lead to equivalent
conditions on the coalgebra.
 In particular, any co-Artinian contramodule is Noetherian, but
the converse is not necessarily true.
 A coalgebra is called \emph{right Artinian} if any finitely
cogenerated right comodule over it is Artinian; this is equivalent to
any finitely generated left contramodule over it being co-Artinian.
 A coalgebra is \emph{right cocoherent} if any finitely cogenerated
quotient comodule of a finitely copresented right comodule over it is
finitely copresented; this is equivalent to any finitely generated
subcontramodule of a finitely presented left contramodule being
finitely presented.

\subsection{{}}
 The finiteness conditions on coalgebras, comodules, and contramodules
are discussed in Section~\ref{coalgebra-finiteness} of
the present paper.
 The definition of a dedualizing complex for a pair of coalgebras is
presented and the triangulated equivalence~\eqref{mgm-equivalence} is
constructed in Section~\ref{dedualizing-bicomodules}.
 The definition of a dedualizing complex for a pair of semialgebras is
spelled out and the triangulated
equivalence~\eqref{semialg-mgm-equivalence} is
constructed in Section~\ref{dedualizing-bisemimodules}.

 We refer to the overview paper~\cite{Prev} and the references therein
for detailed discusions of various kinds of contramodules, including
first of all contramodules over coassociative coalgebras over a field.
 Semialgebras, semimodules, and semicontramodules are discussed
in~\cite[Sections~2.6 and~3.5]{Prev}.
 The structure theory of contramodules over a coalgebra over a field
was studied in~\cite[Appendix~A]{Psemi}.
 The definitions of exotic derived categories used in this paper are
introduced in~\cite[Appendix~A]{Pcosh}; they are also briefly recalled
in~\cite[Appendix~A]{Pmgm}.
 Further discussions can be found in the introductions to~\cite{Pmgm}
and~\cite{Pfp}, and in the references therein.

\medskip
\noindent\textbf{Acknowledgement.}
 The material in
Sections~\ref{coalgebra-finiteness}\+-\ref{dedualizing-bicomodules}
of this paper was originally developed as a part of
the paper~\cite{Pmgm}, but was later excluded from~\cite{Pmgm}
(to make~\cite{Pmgm} a work in commutative algebra) and moved to
this separate paper.
 The author was supported in part by a fellowship from the Lady
Davis Foundation at the Technion while working on~\cite{Pmgm}
(including much of the material now forming the present paper).
 The author's research was supported by the Israel Science Foundation
grant~\#\,446/15 and the Grant Agency of the Czech Republic
under the grant P201/12/G028 while preparing the paper in its
present form.

\Section{Coalgebras with Finiteness Conditions}
\label{coalgebra-finiteness}

 This section contains a discussion of Artinian, co-Noetherian, and
cocoherent coalgebras.
 Many of the results below are certainly not new; we present them
here for the sake of completeness of the exposition.

 We refer to the book~\cite{Swe} and the survey paper~\cite{Prev}
for the definitions of coassociative coalgebras over fields, comodules
and contramodules over them, and the related basic concepts.
 A discussion of \emph{cosemisimple} and \emph{conilpotent} coalgebras
can be found in~\cite[Sections~9.0\+-1]{Swe} and (with a view towards
contramodules and the terminology similar to the one in this paper)
in~\cite[Appendix~A]{Psemi}.

 Let $\C$ be a coassociative coalgebra (with counit) over a field~$k$.
 For any $k$\+vector space $V$ the left $\C$\+comodule $\C\ot_kV$ is
called the \emph{cofree left\/ $\C$\+comodule cogenerated by\/~$V$}.
 For any left $\C$\+comodule $\L$, there is a natural isomorphism
$$
 \Hom_\C(\L\;\C\ot_kV)\simeq\Hom_k(\L,V),
$$
where for any two left $\C$\+comodules $\L$ and $\M$ we denote by
$\Hom_\C(\L,\M)$ the $k$\+vector space of all morphisms $\L\rarrow\M$
in the abelian category of left $\C$\+comodules $\C\comodl$.
 Hence cofree $\C$\+comodules are injective objects in $\C\comodl$.
 Cofree $\C$\+comodules are sufficiently many, so any injective
$\C$\+comodule is a direct summand of a cofree one.
 In particular, the left $\C$\+comodule $\C$ is called
the \emph{cofree\/ $\C$\+comodule with one cogenerator}, and finite
direct sums of copies of $\C$ are the \emph{finitely cogenerated
cofree\/ $\C$\+comodules}.

 A coassociative coalgebra is called \emph{cosimple} if it has no
nonzero proper subcoalgebras.
 The cosimple $k$\+coalgebras are precisely the dual coalgebras to
simple finite-dimensional $k$\+algebras.
 A coassociative coalgebra $\E$ is called \emph{cosemisimple} if it is
a direct sum of cosimple coalgebras, or equivalently, if the category
of left $\E$\+comodules is semisimple, or if the category of right
$\E$\+comodules is semisimple.

 A coassociative coalgebra without counit $\C'$ is called
\emph{conilpotent} if for any element $c'\in\C'$ there exists
an integer $n\ge1$ such that $c'$~is annihilated by the iterated
coaction map $\C'\rarrow\C'{}^{\ot n+1}$.
 Any coassociative coalgebra $\C$ has a unique maximal cosemisimple
subcoalgebra $\C^\ss\subset\C$, which can be also defined as
the (direct) sum of all cosimple subcoalgebras in~$\C$, or as
the minimal subcoalgebra $\E\subset\C$ for which the quotient coalgebra
without counit $\C/\E$ is conilpotent~\cite[Sections~9.0\+-1]{Swe}.
 
 For any subcoalgebra $\E\subset\C$ and any left $\C$\+comodule $\M$,
we denote by ${}_\E\M$ the maximal $\C$\+subcomodule in $\M$ whose
$\C$\+comodule structure comes from an $\E$\+comodule structure.
 In other words, ${}_\E\M\subset\M$ is the full preimage of
the subspace $\E\ot_k\M\subset\C\ot_k\M$ under the left coaction map
$\M\rarrow\C\ot_k\M$.
 The following assertion is a dual version of Nakayama's lemma for
comodules.

\begin{lem} \label{comodule-nakayama}
 Let\/ $\E\subset\C$ be a subcoalgebra such that the quotient
coalgebra without counit\/ $\C/\E$ is conilpotent (i.~e., $\E$
contains the subcoalgebra\/ $\C^\ss\subset\C$).
 Then the subcomodule\/ ${}_\E\M$ is nonzero for any nonzero left\/
$\C$\+comodule\/~$\M$.
\end{lem}

\begin{proof}
 It follows from the conilpotency condition that for every element
$x\in\M$ there exists an integer $n\ge1$ such that $x$~is annihilated
by the iterated coaction map $\M\rarrow(\C/\E)^{\ot n}\ot_k\M$.
 Hence the coaction map $\M\rarrow\C/\E\ot_k\M$ cannot be injective
for a nonzero left $\C$\+comodule~$\M$.
\end{proof}

 A left $\C$\+comodule is said to be
\emph{finitely cogenerated}~\cite[Example~1.2]{Tak} if it can be
embedded as a subcomodule into a finitely cogenerated cofree
left $\C$\+comodule.
 Obviously, any subcomodule of a finitely cogenerated cofree
$\C$\+comodule is finitely cogenerated.
 One easily checks that the class of finitely cogenerated left
$\C$\+comodules is closed under extensions in $\C\comodl$.

\begin{lem} \label{subcoalgebra-finitely-cogenerated}
\textup{(a)} For any finitely cogenerated left\/ $\C$\+comodule $\L$
and any subcoalgebra\/ $\E\subset\C$, the left\/ $\E$\+comodule\/
${}_\E\L$ is finitely cogenerated. \par
\textup{(b)} The cofree left\/ $\C$\+comodule\/ $\C\ot_kV$ with
an infinite-dimensional vector space of cogenerators $V$ over
a nonzero coalgebra\/ $\C$ is not finitely cogenerated. \par
\textup{(c)} For any subcoalgebra\/ $\E\subset\C$, a left\/
$\E$\+comodule\/ $\L$ is finitely cogenerated if and only if it
is finitely cogenerated as a left\/ $\C$\+comodule. \par
\textup{(d)} Let\/ $\E\subset\C$ be a subcoalgebra such that
the quotient coalgebra without counit\/ $\C/\E$ is conilpotent.
 Then a left\/ $\C$\+comodule $\L$ is finitely cogenerated if and only
if the left $\E$\+comodule\/ ${}_\E\L$ is finitely cogenerated. \par
\textup{(e)} A left\/ $\C$\+comodule\/ $\L$ is finitely cogenerated
if and only if the left\/ $\E$\+comodule\/ ${}_\E\L$ for every cosimple
subcoalgebra\/ $\E\subset\C$ is a finite direct sum of copies of
the irreducible left\/ $\E$\+comodule with the multiplicity of
the irreducible left\/ $\E$\+comodule in\/ ${}_\E\L$ divided by its
multiplicity in the left\/ $\E$\+comodule\/ $\E$ bounded by a single
constant uniformly over all the cosimple subcoalgebras\/ $\E\subset\C$.
\end{lem}

\begin{proof}
 Part~(a): obviously, for any injective morphism of left
$\C$\+comodules $\L\rarrow\M$, the induced morphism
${}_\E\L\rarrow{}_\E\M$ is also injective, so it remains to notice
the natural isomorphism of $\E$\+comodules
${}_\E(\C\ot_kV)\simeq\E\ot_kV$ for any $k$\+vector space~$V$.
 Now it suffices to pick any nonzero finite-dimensional subcoalgebra
$\E\subset\C$ in order to deduce part~(b) from the latter isomorphism
and part~(a).
 Part~(c) follows from the same isomorphism.

 Part~(d): a morphism of $\C$\+comodules $\L\rarrow\C\ot_kV$ is
uniquely determined by its composition with the map $\C\ot_kV\rarrow V$
induced by the counit map $\C\rarrow k$ of the coalgebra $\C$; and
this composition can be an arbitrary $k$\+linear map $\L\rarrow V$.
 Suppose that we are given an injective morphism of $\E$\+comodules
${}_\E\L\rarrow\E\ot_kV$, where $V$ is a finite-dimensional
vector space.
 Consider the composition ${}_\E\L\rarrow\E\ot_kV\rarrow V$ and extend
it arbitrarily to a $k$\+linear map $\L\rarrow V$.
 The corresponding $\C$\+comodule morphism $\L\rarrow\C\ot_kV$
forms a commutative diagram with the injective morphism
${}_\E\L\rarrow\E\ot_kV$ and the embeddings ${}_\E\L\rarrow\L$ and
$\E\ot_k V\rarrow\C\ot_kV$.
 Denote by $\K$ the kernel of the morphism $\L\rarrow\C\ot_kV$;
then the submodule $\K\subset\L$ does not intersect the submodule
${}_\E\L\subset\L$, so one has ${}_\E\K=0$.
 By Lemma~\ref{comodule-nakayama}, it follows that $\K=0$.
 To prove part~(e), one applies part~(d) to the subcoalgebra
$\C^\ss\subset\C$ and then decomposes $\C^\ss$ into a direct sum
of its cosimple subcoalgebras~$\E$.
\end{proof}

 A $\C$\+comodule is called \emph{co-Noetherian} if all its
quotient $\C$\+comodules are finitely cogenerated~\cite{WW}.
 The class of co-Noetherian left $\C$\+comodules is closed under
subobjects, quotient objects, and extensions in the abelian category
$\C\comodl$ \cite[Proposition~4]{WW}, so co-Noetherian left
$\C$\+comodules form an abelian category.
 Given a subcoalgebra $\E\subset\C$, an $\E$\+comodule is co-Noetherian
if and only if it is co-Noetherian as a $\C$\+comodule.

 A $\C$\+comodule is called \emph{Artinian} if every descending
chain of its subcomodules terminates.
 As the class of Artinian objects in any abelian category, the class
of Artinian left $\C$\+comodules is closed under subobjects, quotient
objects, and extensions in the abelian category $\C\comodl$, so
Artinian left $\C$\+comodules form an abelian category.
 Given a subcoalgebra $\E\subset\C$, an $\E$\+comodule is Artinian
if and only if it is Artinian as a $\C$\+comodule.

\begin{lem} \label{co-Noetherian-Artinian}
\textup{(a)} Any Artinian\/ $\C$\+comodule is co-Noetherian. \par
\textup{(b)} If the subcoalgebra\/ $\C^\ss\subset\C$ is
finite-dimensional, then any co-Noetherian\/ $\C$\+comodule is
Artinian.
\end{lem}

\begin{proof}
 This is a subset of results of~\cite[Proposition~2.5]{GTNT}.
 Part~(a): it suffices to show that any Artinian left $\C$\+comodule
$\L$ is finitely cogenerated.
 Pick a nonzero linear function $\phi_1\:\L\rarrow k$ and consider
the related morphism of left $\C$\+comodules $f_1\:\L\rarrow\C$.
 Let $\L_1\subset\L$ denote the kernel of the morphism~$f_1$.
 Pick a nonzero linear function $\L_1\rarrow k$ and extend it to
a linear function $\phi_2\:\L\rarrow k$.
 Consider the related morphism of left $\C$\+comodules $f_2\:
\L\rarrow\C$; let $\L_2\subset\L_1$ denote the intersection of
the kernels of the morphisms $f_1$ and~$f_2$, etc.
 According to the descending chain condition, this process must
terminate, which can only happen if the intersection of the kernels
of the morphisms $f_1$,~\dots, $f_n$ is zero for some integer~$n$.
 We have constructed an injective morphism of left $\C$\+comodules
$\L\rarrow \C^{\oplus n}$.

 Part~(b): it suffices to show that any descending chain of
subcomodules $\L\supset\L_1\supset\L_2\supset\dotsb$ with zero
intersection $\bigcap_n\L_n=0$ terminates in a finitely cogenerated
left $\C$\+comodule~$\L$.
 Indeed, by Lemma~\ref{subcoalgebra-finitely-cogenerated}(a)
together with the assumption of part~(b) the subcomodule
${}_{\C^\ss}\L\subset\L$ is finite-dimensional.
 Hence the chain of intersections ${}_{\C^\ss}\L\cap\L_i$ stabilizes,
and consequently, eventually vanishes, i.~e., there exists~$n$
for which ${}_{\C^\ss}\L\cap\L_n=0$.
 Then it follows from Lemma~\ref{comodule-nakayama} that $\L_n=0$.
\end{proof}

 A \emph{left contramodule} $\P$ over a coassociative coalgebra~$\D$
over a field~$k$ is a $k$\+vector space endowed with a \emph{left\/
$\D$\+contraaction} map $\Hom_k(\D,\P)\rarrow\P$ satisfying
the appropriate \emph{contraassociativity} and \emph{contraunitality}
equations.
 Specifically, the two maps $\Hom_k(\D,\Hom_k(\D,\P)\simeq
\Hom_k(\D\ot_k\D\;\P)\birarrow\Hom_k(\D,\P)$ induced by
the comultiplication map $\D\rarrow\D\ot_k\D$ and the contraaction
map should have equal compositions with the contraaction map
$\Hom_k(\D,\P)\rarrow\P$,
$$
 \Hom_k(\D,\Hom_k(\D,\P))\simeq\Hom_k(\D\ot_k\D,\P)\birarrow
 \Hom_k(\D,\P)\rarrow\P,
$$
while the composition of the map $\P\rarrow\Hom_k(\D,\P)$ induced
by the counit map $\D\rarrow k$ with the contraaction map should
be equal to the identity map on the contramodule~$\P$,
$$
 \P\rarrow\Hom_k(\D,\P)\rarrow\P.
$$
 The natural isomorphism $\Hom_k(U,\Hom_k(V,W))\simeq
\Hom_k(V\ot_kU\;W)$ connecting the tensor product and Hom functors
on the category of $k$\+vector spaces is presumed in the first
equation.

 Left $\D$\+contramodules form an abelian category $\D\contra$
with an exact forgetful functor to the category of $k$\+vector
spaces $\D\contra\rarrow k\vect$, preserving infinite products
but not infinite direct sums (see~\cite[Sections~1.1\+-1.2]{Prev}
and the references therein).
 For any right $\D$\+comodule $\N$ and $k$\+vector space $V$,
the vector space $\Hom_k(\N,V)$ has a natural left
$\D$\+contramodule structure.
 In particular, the left $\D$\+contramodule $\Hom_k(\D,V)$ is called
the \emph{free left\/ $\D$\+contramodule generated by\/~$V$}.
 For any left $\D$\+contramodule~$\Q$, there is a natural isomorphism
$$
 \Hom^\D(\Hom_k(\D,V),\Q)\simeq\Hom_k(V,\Q),
$$
where for any two left $\D$\+contramodules $\P$ and $\Q$ we denote
by $\Hom^\D(\P,\Q)$ the $k$\+vector space of all morphisms
$\P\rarrow\Q$ in the abelian category $\D\contra$.
 Hence free $\D$\+contramodules are projective objects in $\D\contra$.
 There are enough of them, so any projective left $\D$\+contramodule
is a direct summand of a free one.
 The left $\D$\+contramodule $\Hom_k(\D,k)$ is called
the \emph{free\/ $\D$\+contramodule with one generator}, and
the $\D$\+contramodules $\Hom_k(\D,V)$ with finite-dimensional
$k$\+vector spaces $V$ are the \emph{finitely generated free\/
$\D$\+contramodules}.

 For any subcoalgebra $\E\subset\D$ and any left $\D$\+contramodule
$\P$, we denote by ${}^\E\P$ the maximal quotient $\D$\+contramodule
of $\P$ whose $\D$\+contramodule structure comes from
an $\E$\+contramodule structure.
 In other words, ${}^\E\P$ is the cokernel of the composition
$\Hom_k(\D/\E,\P)\rarrow\P$ of the embedding $\Hom_k(\D/\E,\P)\rarrow
\Hom_k(\D,\P)$ with the contraaction map $\Hom_k(\D,\P)\rarrow\P$.
 The following assertion is called the \emph{Nakayama lemma for
contramodules} over coalgebras over fields.

\begin{lem} \label{contramodule-nakayama}
 Let\/ $\E\subset\D$ be a subcoalgebra such that the quotient coalgebra
without counit\/ $\D/\E$ is conilpotent.
 Then the quotient contramodule ${}^\E\P$ is nonzero for any nonzero
left\/ $\D$\+contramodule\/~$\P$.
\end{lem}

\begin{proof}
 This is~\cite[Lemma~A.2.1]{Psemi}; see also~\cite[Lemma~1.3.1]{Pweak}
and~\cite[Lemma~2.1]{Prev}.
\end{proof}

 A left $\D$\+contramodule is said to be \emph{finitely generated}
if it is a quotient contramodule of a finitely generated free left
$\D$\+contramodule.
 The class of finitely generated left $\D$\+contramodules is closed
under extensions and the passages to quotient objects.

\begin{lem} \label{subcoalgebra-contramod-finitely-generated}
\textup{(a)} For any finitely generated left\/ $\D$\+contramodule\/
$\Q$ and any subcoalgebra\/ $\E\subset\D$, the left\/
$\E$\+contramodule\/ ${}^\E\Q$ is finitely generated. \par
\textup{(b)} The free\/ $\D$\+contramodule $\Hom_k(\D,V)$ with
an infinite-dimensional vector space of generators $V$ over a nonzero
coalgebra\/ $\D$ is not finitely generated. \par
\textup{(c)} For any subcoalgebra\/ $\E\subset\D$, a left\/
$\E$\+contramodule is finitely generated if and only if it is
finitely generated as a left\/ $\D$\+contramodule. \par
\textup{(d)} Let\/ $\E\subset\D$ be a subcoalgebra such that
the quotient coalgebra without counit\/ $\D/\E$ is conilpotent.
 Then a left\/ $\D$\+contramodule\/ $\Q$ is finitely generated if
and only if the left\/ $\E$\+contramodule\/ ${}^\E\Q$ is finitely
generated. \par
\textup{(e)} A left\/ $\D$\+contramodule\/ $\Q$ is finitely generated
if and only if the left\/ $\E$\+contra\-module\/ ${}^\E\Q$ for every
cosimple subcoalgebra\/ $\E\subset\D$ is a finite direct sum of copies
of the irreductive left\/ $\E$\+contramodule with the multiplicity of
the irreductible left\/ $\E$\+contra\-module in\/ ${}^\E\Q$ divided by
its multiplicity in the left\/ $\E$\+contramodule\/ $\E^*=\Hom_k(\E,k)$
bounded by a single constant uniforly over all the cosimple subcoalgebras
$\E\subset\D$.
\end{lem}

\begin{proof}
 The proof is dual-analogous to that of
Lemma~\ref{subcoalgebra-finitely-cogenerated}.
 To prove parts~(a\+c), one notices the natural isomorphism
${}^\E\Hom_k(\D,V)\simeq\Hom_k(\E,V)$ for any subcoalgebra
$\E\subset\D$ and $k$\+vector space~$V$.
 Part~(d): given a surjective morphism of $\E$\+contramodules
$\Hom_k(\E,V)\lrarrow{}^\E\Q$ with a finite-dimensional vector
space~$V$, one considers the composition $V\rarrow\Hom_k(\E,V)\rarrow
{}^\E\Q$ and lifts it to a $k$\+linear map $V\rarrow\Q$.
 The corresponding morphism of $\D$\+contramodules $\Hom_k(\D,V)
\rarrow\Q$ is surjective by Lemma~\ref{contramodule-nakayama}, since
one has ${}^\E\fK=0$ for its cokernel~$\fK$.
 To prove part~(e), one applies part~(d) to the subcoalgebra
$\D^\ss\subset\D$ and applies~\cite[Lemma~A.2.2]{Psemi} in order to
decompose the $\D^\ss$\+contramodule ${}^{\D^\ss}\Q$ into a product of
contramodules over the cosimple subcoalgebras $\E\subset\D^\ss$.
\end{proof}

 A left $\D$\+contramodule is called \emph{Noetherian} if all its
subcontramodules are finitely generated.
 The class of Noetherian left $\D$\+contramodules is closed under
subobjects, quotient objects, and extensions in the abelian category
$\D\contra$, so Noetherian left $\D$\+contramodules form an abelian
category.
 Given a subcoalgebra $\E\subset\D$, an $\E$\+contramodule is
Noetherian if and only if it is Noetherian as a $\D$\+contramodule.

 A $\D$\+contramodule is called \emph{co-Artinian} if every ascending
chain of its subcontramodules terminates.
 As the similar class of objects in any abelian category, the class of
co-Artinian left $\D$\+contramodules is closed under subobjects,
quotient objects, and extensions in the abelian category $\D\contra$,
so co-Artinian left $\D$\+contramodules form an abelian category.
 Given a subcoalgebra $\E\subset\D$, an $\E$\+contramodule is
co-Artinian if and only if it is co-Artinian as a $\D$\+contramodule.

\begin{lem}
\textup{(a)} Any co-Artinian $\D$\+contramodule is Noetherian. \par
\textup{(b)} If the subcoalgebra $\D^\ss\subset\D$ is
finite-dimensional, then any Noetherian $\D$\+contramodule is
co-Artinian.
\end{lem}

\begin{proof}
 Part~(a): it suffices to show that any co-Artinian left
$\D$\+contramodule $\Q$ is finitely generated.
 Pick an element $q_1\in\Q$ and consider the related morphism of
left $\D$\+contramodules $f_1\:\D^*=\Hom_k(\D,k)\rarrow\Q$.
 Pick an element $q_2\in\Q$ outside of the image of~$f_1$, consider
the related morphism $f_2\:\D^*\rarrow\Q$, pick an element $q_3\in\Q$
outside of the sum of the images of $f_1$ and~$f_2$, etc.
 According to the ascending chain condition, this process must
terminate, which means that the sum of the images of
the morphisms~$f_1$,~\dots, $f_n$ is the whole of $\Q$ for some
integer~$n$.
 We have constructed as surjective morphism of left $\D$\+contramodules
$\D^*{}^{\oplus n}\rarrow\Q$.

 Part~(b): it suffices to show that an ascending chain of
subcontramodules $\Q_1\subset\Q_2\subset\dotsb\subset\Q$ terminates
in a finitely generated left $\D$\+contramodule $\Q$ provided that
there is no proper subcontramodule in $\Q$ containing all
the subcontramodules~$\Q_n$.
 Indeed, by Lemma~\ref{subcoalgebra-contramod-finitely-generated}(a)
together with the assumption of part~(b) the maximal quotient
$\D^\ss$\+contramodule ${}^{\D^\ss}\Q$ of the $\D$\+contramodule $\Q$
is finite-dimensional.
 Hence the chain of the images of the subcontramodules $\Q_n\subset\Q$
in ${}^{\D^\ss}\Q$ stabilizes, and consequently, eventually reaches
the whole of ${}^{\D^\ss}\Q$, i.~e., there exists~$n$ for which
the composition $\Q_n\rarrow\Q\rarrow{}^{\D^\ss}\Q$ is surjective.
 Then one has ${}^{\D^\ss}(\Q/\Q_n)=0$, and it follows from
Lemma~\ref{contramodule-nakayama} that $\Q_n=\Q$.
\end{proof}

\begin{ex}  \label{co-Noetherian-Artinian-counterex}
 Let $\C$ be an infinite-dimensional cosemisimple coalgebra.
 Then the left $\C$\+comodule $\C$ is co-Noetherian, but not Artinian.
 Similarly, the left $\C$\+contra\-module $\C^*$ is Noetherian, but
not co-Artinian.
 It follows that the classes of Artinian and co-Noetherian left
comodules over a coalgebra $\D$ coincide if and only if the classes of
co-Artinian and Noetherian left contramodules over $\D$ coincide and
if and only if the subcoalgebra $\D^\ss\subset\D$ is finite-dimensional.
\end{ex}

 A left $\D$\+contramodule is said to be \emph{finitely presented} if
it is the cokernel of a morphism of finitely generated free left
$\D$\+contramodules.
 Clearly, the cokernel of a morphism from a finitely generated left
$\D$\+contramodule to a finitely presented one is finitely presented.
 It is easy to check that an extension of finitely presented
left $\D$\+contramodules is finitely presented.

 A left $\C$\+comodule is said to be \emph{finitely copresented} if
it is the kernel of a morphism of finitely cogenerated cofree
$\C$\+comodules.
 Clearly, the kernel of a morphism from a finitely copresented left
$\C$\+comodule to a finitely cogenerated one is finitely copresented;
an extension of finitely copresented left $\C$\+comodules is finitely
copresented.

 Part~(a) of the next lemma can be found in~\cite[Theorem~6]{WW}.

\begin{lem} \label{finitely-presented}
\textup{(a)} The cokernel of an injective morphism from a finitely
copresented\/ $\C$\+comodule to a finitely cogenerated one is
finitely cogenerated. \par
\textup{(b)} The kernel of a surjective morphism from a finitely
generated\/ $\D$\+contramodule to a finitely presented one is
finitely generated.
\end{lem}

\begin{proof}
 Part~(a): let $\L$ be the kernel of a morphism of finitely cogenerated
cofree $\C$\+comodules $\I\rarrow\J$, let $\M$ be a finitely
cogenerated $\C$\+comodule, and let $\L\rarrow\M$ be an injective
morphism with the cokernel~$\K$.
 Denote by $\N$ the fibered coproduct of $\C$\+comodules $\I$ and
$\M$ over the $\C$\+comodule~$\L$; then there are exact sequences
of $\C$\+comodules $0\rarrow\M\rarrow\N\rarrow\J$ and $0\rarrow\I
\rarrow\N\rarrow\K\rarrow0$.
 Now the $\C$\+comodule $\N$ is finitely cogenerated as an extension
of finitely cogenerated $\C$\+comodules; and the $\C$\+comodule $\K$ is
a direct summand of $\N$, because the $\C$\+comodule $\I$ is injective.
 The proof of part~(b) is analogous.
\end{proof}

 The dual vector space $\D^*$ to a coassociative coalgebra $\D$ has
a natural structure of topological associative algebra.
 There is but a slight ambiguity in its definition in that one has
to make a decision about the order of the factors in the multiplication
operation, i.~e., which one of the two opposite algebras is to be
denoted by $\D^*$ and which one by~$\D^*{}^\rop$.
 We prefer the convention according to which right $\D$\+comodules $\N$
become discrete right $\D^*$\+modules; then the dual vector space
$\N^*$ is a left $\D^*$\+module (see~\cite[Sections~1.3\+-4]{Prev} for
a further discussion).
 Any left $\D$\+contramodule has an underlying structure of left
$\D^*$\+module (see~\cite[Section~2.3]{Prev}
and~\cite[Section~A.1.2]{Psemi}).

 One observes that a left $\D$\+contramodule is finitely generated if
and only if its underlying left $\D^*$\+module is finitely generated.
 It follows that a left $\D$\+contramodule is finitely presented if
and only if its underlying left $\D^*$\+module is.

\begin{prop} \label{finitely-co-presented-hom}
\textup{(a)} The restrictions of the functor\/ $\L\longmapsto\L^*=
\Hom_k(\L,k)$ and the forgetful functor\/ $\D\contra\rarrow\D^*\modl$
provide an anti-equivalence between the additive category of finitely
copresented right\/ $\D$\+comodules and the additive category of
finitely presented left\/ $\D$\+contramodules, and an isomorphism
between the latter category and the additive category of finitely
presented left\/ $\D^*$\+modules. \par
\textup{(b)} For any right\/ $\D$\+comodule\/ $\N$ and any finitely
copresented right\/ $\D$\+comodule\/ $\L$, the functor\/
$\N\longmapsto\N^*=\Hom_k(\N,k)$ and the forgetful functor\/
$\D\contra\rarrow\D^*\modl$ induce isomorphisms of the Hom spaces
$$
 \Hom_{\D^\rop}(\N,\L)\simeq\Hom^\D(\L^*,\N^*)\simeq
 \Hom_{\D^*}(\L^*,\N^*)
$$
in the categories of right\/ $\D$\+comodules, left\/
$\D$\+contramodules, and left\/ $\D^*$\+modules.
\end{prop}

\begin{proof}
 Since the functor $\Hom$ preserves kernels in its second argument
and transforms cokernels in its first argument into kernels, it
suffices to prove part~(b) for finitely generated cofree right
$\D$\+comodules $\L=V\ot_k\D$, where $V$ is a finite-dimensional
$k$\+vector space.
 Then $\L^*\simeq\Hom_k(\D,V^*)\simeq \D^*\ot_k V^*$ is a finitely
generated free left $\D$\+contramodule and a finitely generated
free left $\D^*$\+module.
 One easily computes $\Hom_{\D^\rop}(\N\;V\ot_k\D)\simeq\Hom_k(\N,V)$,
\ $\Hom^\D(\Hom_k(\D,V^*),\N^*)\simeq\Hom_k(V^*,\N^*)$, and
$\Hom_{\D^*}(\D^*\ot_kV^*\;\N^*)\simeq\Hom_k(V^*,\N^*)$,
implying part~(b).
 Part~(a) immediately follows from the same computation of Hom spaces.
\end{proof}

\begin{lem} \label{co-contra-Artinian-Noetherian}
\textup{(a)} A left\/ $\D$\+contramodule is co-Artinian if and only if
it is a Noetherian left\/ $\D^*$\+module. \par
\textup{(b)} A right\/ $\D$\+comodule\/ $\L$ is Artinian if and only if
dual vector space\/ $\L^*$ is a Noetherian left\/ $\D^*$\+module. \par
\textup{(c)} A right\/ $\D$\+comodule\/ $\L$ is co-Noetherian provided
that its dual vector space\/ $\L^*$ is a Noetherian left\/
$\D$\+contramodule.
\end{lem}

\begin{proof}
 Part~(a): one notices that a $\D^*$\+module is Noetherian if and
only if any ascending chain of its finitely generated submodules
terminates.
 Similarly, a $\D$\+contramod\-ule is co-Artinian if and only if
any ascending chain of its finitely generated subcontramodules
terminates.
 Finally, the classes of finitely generated $\D^*$\+submodules and
finitely generated $\D$\+subcontramodules in any given
$\D$\+contramodule coincide.

 Part~(b) is again a subset of~\cite[Proposition~2.5]{GTNT}.
 To any descending chain of $\D$\+sub\-comodules in $\L$ one can
assign the ascending chain of their orthogonal complements, which are
$\D^*$\+submodules in~$\L^*$.
 Conversely, in view of Proposition~\ref{finitely-co-presented-hom}(b),
any finitely generated $\D^*$\+submodule in $\L^*$ is the orthogonal
complement to a certain $\D$\+subcomodule in~$\L$.
 Part~(c): for any quotient comodule of $\L$, there is its dual
subcontramodule in~$\L^*$.
 It remains to notice that a right $\D$\+comodule $\N$ is finitely
cogenerated if and only if its dual left $\D$\+contramodule $\N^*$
is finitely generated.
\end{proof}

 A finitely cogenerated left $\C$\+comodule is called \emph{cocoherent}
if every its finitely cogenerated quotient comodule is finitely
copresented.
 Using Lemma~\ref{finitely-presented}(a), one can show that the class
of cocoherent left $\C$\+comodules is closed under the operations of
the passage to the kernels, cokernels, and extensions in the abelian
category $\C\comodl$; so cocoherent left $\C$\+comodules form
an abelian category.

 Analogously, a finitely presented left $\D$\+contramodule is called
\emph{coherent} if every its finitely generated subcontramodule is
finitely presented.
 Using Lemma~\ref{finitely-presented}(b), one shows that the class of
coherent left $\D$\+contramodules is closed under the passages to
the kernels, cokernels, and extensions in the abelian category
$\D\comodl$; so coherent left $\D$\+contramodules form
an abelian category.

\begin{lem} \label{co-contra-coherent}
\textup{(a)} A left\/ $\D$\+contramodule is coherent if and only if its
underlying left\/ $\D^*$\+module is coherent.
 The abelian categories of coherent left\/ $\D$\+contramodules and
coherent left\/ $\D^*$\+modules are isomorphic. \par
\textup{(b)} A right\/ $\D$\+comodule\/ $\L$ is cocoherent if and only
if its dual left\/ $\D^*$\+module\/ $\L^*$ is coherent.
 The abelian categories of cocoherent right\/ $\D$\+comodules and
coherent left\/ $\D^*$\+modules are anti-equivalent.  
\end{lem}

\begin{proof}
 In view of Proposition~\ref{finitely-co-presented-hom}(a), it suffices
to check the first assertion in each of the parts~(a) and~(b).
 In part~(a), one uses the bijection between finitely generated
$\D$\+subcontramodules and finitely generated $\D^*$\+submodules of
a given $\D$\+contra\-module, together with the fact that
a $\D$\+contramodule is finitely presented if and only if it is
finitely presented as a $\D^*$\+module.
 In part~(b), one uses the bijection between finitely cogenerated
quotient $\D$\+comodules of $\L$ and finitely generated
$\D^*$\+submodules of $\L^*$, together with the fact that
a $\D$\+comodule is finitely copresented if and only if its dual
$\D^*$\+module is finitely presented.
\end{proof}

 A coalgebra $\C$ is called \emph{left co-Noetherian} if any quotient
comodule of a finitely cogenerated left $\C$\+comodule is finitely
cogenerated, or equivalently, if the left $\C$\+comod\-ule $\C$ is
co-Noetherian~\cite[Theorem~3]{WW}.
 Over a left co-Noetherian coalgebra $\C$, finitely cogenerated left
comodules form an abelian category.
 By Lemma~\ref{subcoalgebra-finitely-cogenerated}(c), any subcoalgebra
of a left co-Noetherian coalgebra is left co-Noetherian.
 Any cosemisimple coalgebra is left and right co-Noetherian.

 A coalgebra $\D$ is called \emph{right Artinian} if any finitely
cogenerated right $\D$\+comodule is Artinian, or equivalently, if
the right $\D$\+comodule $\D$ is Artinian, or if any finitely
generated left $\D$\+contramodule is co-Artinian, or if the left
$\D$\+contramodule $\D^*=\Hom_k(\D,k)$ is co-Artinian (see
Lemma~\ref{co-contra-Artinian-Noetherian}(a\+b) for a proof of
the equivalence between the second and the fourth of these conditions).
 A coalgebra $\D$ is right Artinian if and only if its dual algebra
$\D^*$ is left Noetherian.
 Any subcoalgebra of a right Artinian coalgebra is right Artinian.

 According to Lemma~\ref{co-Noetherian-Artinian} and
Example~\ref{co-Noetherian-Artinian-counterex}, any left Artinian
coalgebra $\C$ is left co-Noetherian, but the converse is not
generally true.
 More precisely, a coalgebra $\C$ is left Artinian if and only if it is
left co-Noetherian and its maximal cosemisimple subcoalgebra
$\C^\ss\subset\C$ is finite-dimensional.
 
\begin{exs} \label{dual-to-taylor-series}
 The functor $\C\longmapsto\C^*$ is an anti-equivalence between
the category of coassociative coalgebras and the category of
pro-finite-dimensional topological associative algebras, so one
can describe coalgebras in terms of their dual topological algebras.
 In particular, the topological algebra of formal Taylor power series
in commuting variables $k[[z_1,\dotsc,z_m]]$ corresponds to a certain
cocommutative coalgebra~$\C$.
 The algebra $k[[z_1,\dotsc,z_m]]$ is Noetherian, so the coalgebra
$\C$ is Artinian.
 Hence all the subcoalgebras of $\C$ are Artinian (and consequently,
co-Noetherian), too.
 These are precisely the coalgebras dual to the topological algebras
of functions on the formal completions of algebraic varieties
over~$k$ at their closed points defined over~$k$.
 Given a field extension $k\subset\ell$, a coalgebra $\C$ over
the field~$k$ is Artinian or co-Noetherian whenever the coalgebra
$\ell\ot_k\C$ over the field~$\ell$ is.
 Hence it follows that all the coalgebras dual to the topological
algebras of functions on the formal completions of varieties over~$k$
at their closed points are Artinian.

 Moreover, there are many noncocommutative Artinian coalgebras, like,
e.~g., the coalgebra dual to the algebra of quantum formal power series
$k\{\{z_1,\dotsb,z_m\}\}$ with the relations $z_iz_j=q_{i,j}z_jz_i$
for all $i<j$, with any constants $q_{i,j}\in k^*$.
\end{exs}

 A coalgebra $\D$ is called \emph{right cocoherent} if any finitely
cogenerated quotient comodule of a finitely copresented right
$\D$\+comodule is finitely copresented, or equivalently, if the right
$\D$\+comodule $\D$ is cocoherent.
 Equivalently, a coalgebra $\D$ is right cocoherent if any finitely
generated subcontramodule of a finitely presented left
$\D$\+contramodule is finitely presented, or if the left
$\D$\+contramodule $\D^*$ is coherent.
 Over a right cocoherent coalgebra $\D$, both the finitely copresented
right $\D$\+comodules and the finitely presented left
$\D$\+contramodules form abelian categories.
 A coalgebra $\D$ is right cocoherent if and only if its dual
algebra $\D^*$ is left coherent (see Lemma~\ref{co-contra-coherent}).
 Any left co-Noetherian coalgebra $\C$ is left cocoherent, and any
finitely cogenerated left $\C$\+comodule is finitely copresented.

\bigskip

 The \emph{contratensor product} $\N\ocn_\D\P$ of a right
$\D$\+comodule $\N$ and a left $\D$\+contra\-module $\P$
\cite[Section~3.1]{Prev} is a $k$\+vector space constructed as
the cokernel of (the difference of) the pair of maps
$$
 \N\ot_k\Hom_k(\D,\P)\birarrow\N\ot_k\P,
$$
one which is induced by the $\D$\+contraaction in $\P$, while the other
one is the composition $\N\ot_k\Hom_k(\D,\P)\rarrow\N\ot_k\D\ot_k
\Hom_k(\D,\P)\rarrow\N\ot_k\P$ of the map induced by the right
$\D$\+coaction map $\N\rarrow\N\ot_k\D$ and the map induced by
the evaluation map $\D\ot_k\Hom_k(\D,\P)\rarrow\P$.
 The functor of contratensor product of comodules and contramodules
over a coalgebra~$\D$ is right exact.

 For any right $\D$\+comodule $\N$ and a $k$\+vector space $V$ there is
a natural isomorphism of $k$\+vector spaces
$$
 \N\ocn_\D\Hom_k(\D,V)\simeq\N\ot_kV,
$$
while for any right $\D$\+comodule $\N$, any left $\D$\+contramodule
$\P$, and a $k$\+vector space $V$ there is a natural isomorphism of
$k$\+vector spaces
$$
 \Hom_k(\N\ocn_\D\P\;V)\simeq\Hom^\D(\P,\Hom_k(\N,V)).
$$

 The \emph{cotensor product} $\N\oc_\C\M$ of a right $\C$\+comodule
$\N$ and a left $\C$\+comodule $\M$ \cite[Sections~2.5\+-6]{Prev} is
a $k$\+vector space constructed as the kernel of the pair of maps
$$
 \N\ot_k\M\birarrow\N\ot_k\C\ot_k\M,
$$
one of which is induced by the right $\C$\+coaction in $\N$ and
the other one by the left $\C$\+coaction in~$\M$.
 The functor of cotensor product of comodules over a coalgebra $\C$
is left exact.

 For any right $\C$\+comodule $\N$, left $\C$\+comodule $\M$, and
$k$\+vector space $V$ there are natural isomorphisms of $k$\+vector
spaces
$$
 \N\oc_\C(\C\ot_k V)\simeq \N\ot_k V
 \quad\text{and}\quad
 (V\ot_k\C)\oc_\C\M \simeq V\ot_k \M.
$$
 For any left $\C$\+comodule $\M$ and any subcoalgebra $\E\subset\C$
there is a natural isomorphism of left $\E$\+comodules
$$
 {}_\E\M\simeq\E\oc_\C\M,
$$
where the left $\E$\+comodule structure on the cotensor product is
induced by the left $\E$\+comodule structure on~$\E$.

 The $k$\+vector space of \emph{cohomomorphisms} $\Cohom_\D(\M,\P)$
from a left $\C$\+comodule $\M$ to a left $\D$\+contramodule $\P$ is
a $k$\+vector space constructed as the cokernel of the pair of maps
$$
 \Hom_k(\D\ot_k\M\;\P)\simeq\Hom_k(\M,\Hom_k(\D,\P))
 \birarrow\Hom_k(\M,\P),
$$
one of which is induced by the left $\D$\+coaction in $\M$ and
the other one by the left $\D$\+contraaction in $\P$.
 The functor of cohomomorphisms from left comodules to left
contramodules over a coalgebra $\D$ is right exact.

 For any left $\D$\+comodule $\M$, left $\D$\+contramodule $\P$, and
$k$\+vector space $V$ there are natural isomorphisms of $k$\+vector
spaces
$$
 \Cohom_\D(\D\ot_kV\;\P)\simeq\Hom_k(V,\P)
 \ \ \text{and}\ \ 
 \Cohom_\D(\M,\Hom_k(\D,V))\simeq\Hom_k(\M,V).
$$
 For any left $\D$\+contramodule $\P$ and any subcoalgebra
$\E\subset\D$ there is a natural isomorphism of left
$\E$\+contramodules
$$
 {}^\E\P\simeq\Cohom_\D(\E,\P),
$$
where the left $\E$\+contramodule structure on the Cohom space is
induced by the right $\E$\+comodule structure on~$\E$.

\begin{lem} \label{tensor-contratensor}
 For any right\/ $\D$\+comodule $\N$ and any left\/
$\D$\+contramodule\/ $\P$ there is a natural surjective map of
$k$\+vector spaces from the tensor product over the algebra\/
$\D^*$ to the contratensor product over the coalgebra\/~$\D$
$$
 \N\ot_{\D^*}\P\lrarrow\N\ocn_\D\P.
$$
 This map is an isomorphism, at least, whenever either \par
\textup{(a)} the left\/ $\D$\+contramodule\/ $\P$ is finitely presented,
or \par
\textup{(b)} the coalgebra\/ $\D$ is left co-Noetherian.
\end{lem}

\begin{proof}
 To construct the surjective $k$\+linear map in question, one notices
that the tensor product $\N\ot_{\D^*}\P$ is the cokernel of a natural
map $\N\ot_k\D^*\ot_k\P\rarrow\N\ot_k\P$, while the contratensor
product $\N\oc_\D\P$ is the cokernel of a map
$\N\ot_k\Hom_k(\D,\P)\rarrow\N\ot_k\P$.
 These two maps form a commutative diagram with the natural embedding
$$
 \N\ot_k\D^*\ot_k\P\lrarrow\N\ot_k\Hom_k(\D,\P).
$$
 To prove part~(a), one considers the induced map of the dual
vector spaces
$$
 (\N\oc_\D\P)^*\simeq\Hom^\D(\P,\N^*)\lrarrow
 \Hom_{\D^*}(\P,\N^*)\simeq(\N\ot_{\D^*}\P)^*
$$
and applies Proposition~\ref{finitely-co-presented-hom}.

 To prove part~(b), notice that any right $\D$\+comodule $\N$ is
the union of its maximal $\E$\+subcomodules $\N_\E$ over all
the finite-dimensional subcoalgebras $\E\subset\D$.
 Since both the tensor and the contratensor products preserve
inductive limits in their first arguments, it suffices to
consider the case of a right $\E$\+comodule $\N=\N_\E$.
 Then one has $\N\ot_{\D^*}\P\simeq\N\ot_{\E^*}(\E^*\ot_{\D^*}\P)$
and $\N\ocn_\D\P\simeq\N\ot_{\E^*}{}^\E\P$, so it remains to show
that the natural map
$$
 \E^*\ot_{\D^*}\P\lrarrow
 {}^\E\P\simeq\E^*\ocn_\D\P\simeq\Cohom_\D(\E,\P)
$$
is an isomorphism.
 For this purpose, one presents the left $\D$\+comodule $\E$ as
the kernel of a morphism of finitely cogenerated cofree left
$\D$\+comodules and uses the right exactness property of
the functor $\Cohom_\D$ together with the natural isomorphism
$$
 \J^*\ot_{\D^*}\P=(V^*\ot_k\D^*)\ot_{\D^*}\P\.\simeq\.
 \Cohom_\D(\D\ot_kV\;\P)=\Cohom_\D(\J,\P)
$$
for a finitely cogenerated cofree left $\D$\+comodule $\J=\D\ot_kV$
and any left $\D$\+contra\-module~$\P$.
\end{proof}

\Section{MGM Duality for Coalgebras}
\label{dedualizing-bicomodules}

 We start with several constructions and lemmas related to complexes of
comodules and contramodules.
 These are purported to clear way to our key definition of
a \emph{dedualizing complex} of bicomodules over a pair of
cocoherent coalgebras.

 Let $\C$ and $\D$ be two coassociative coalgebras (with counits)
over the same field~$k$.
 Given a derived category symbol $\st=\b$, $+$, $-$, $\varnothing$,
$\abs+$, $\abs-$, or~$\abs$, we denote by $\sD^\st(\C\comodl)$ and
$\sD^\st(\D\contra)$ the corresponding (conventional or absolute)
derived categories of the abelian categories $\C\comodl$ and
$\D\contra$ of left $\C$\+comodules and left $\D$\+contramodules
(see~\cite[Appendix~A]{Pcosh} or~\cite[Appendix~A]{Pmgm}
for the definitions).

 For any subcoalgebra $\E$ in a coalgebra $\C$, the maximal
$\E$\+subcomodule functor $\M\longmapsto{}_\E\M$ acting from
the category of left $\C$\+comodules to the category of left
$\E$\+comodules is left exact.
 The abelian category $\C\comodl$ has enough injective objects, which
are precisely the direct summands of cofree $\C$\+comodules.
 So one can identify the bounded below derived category
$\sD^+(\C\comodl)$ with the homotopy category of injective
$\C$\+comodules $\Hot^+(\C\comodl_\inj)$ and, applying the functor
$\M\longmapsto{}_\E\M$ to complexes of injective $\C$\+comodules
termwise, obtain the right derived functor
$$
 \M^\bu\longmapsto{}_\E^\boR\M^\bu\:\,
 \sD^+(\C\comodl)\lrarrow\sD^+(\E\comodl).
$$

\begin{lem}  \label{finitely-cogenerated-complex}
 Let\/ $\C$ be a left co-Noetherian coalgebra and\/ $\E\subset\C$ be
a subcoalgebra such that the quotient coalgebra without counit\/
$\C/\E$ is conilpotent.
 Then a complex\/ $\L^\bu\in\sD^+(\C\comodl)$ has finitely
cogenerated\/ $\C$\+comodules of cohomology if and only if
the complex\/ ${}_\E^\boR\L^\bu\in\sD^+(\E\comodl)$ has
finitely cogenerated\/ $\E$\+comodules of cohomology.
\end{lem}

\begin{proof}
 Notice that the cohomology $\E$\+comodules of the complex
${}_\E^\boR\L$ are finitely cogenerated for any finitely cogenerated
left $\C$\+comodule $\L$ (viewed as a one-term complex of left
$\C$\+comodules).
 Indeed, one can compute the derived category object ${}_\E^\boR\L$
using a right resolution of the $\C$\+comodule $\L$ by finitely
cogenerated cofree left $\C$\+comodules (which exists since the class
of finitely cogenerated left comodules over a left co-Noetherian
coalgebra $\C$ is closed under the passages to the cokernels of
morphisms) and apply Lemma~\ref{subcoalgebra-finitely-cogenerated}(a).
 Since the class of finitely cogenerated left $\E$\+comodules is
also closed under the kernels, cokernels, and extensions,
the desired assertion now follows by induction in the cohomological
degree from Lemma~\ref{subcoalgebra-finitely-cogenerated}(d).
\end{proof}

 We recall from Section~\ref{coalgebra-finiteness} that finitely
copresented left comodules over a left cocoherent coalgebra $\C$ form
an abelian category.
 Notice that this abelian category has enough injective objects, which
are precisely the direct summands of finitely cogenerated cofree
$\C$\+comodules.

\begin{lem}  \label{finitely-cogenerated-cofree-complex}
 Let\/ $\C$ be a left cocoherent coalgebra, and let\/ $\L^\bu$ be
a bounded below complex of left\/ $\C$\+comodules with finitely
copresented cohomology modules.
 Then there exists a bounded below complex of finitely cogenerated
cofree left\/ $\C$\+comodules\/ $\J^\bu$ together with
a quasi-isomorphism of complexes of left\/ $\C$\+comodules\/
$\L^\bu\rarrow\J^\bu$.
\end{lem}

\begin{proof}
 This is a standard step-by-step construction
(cf.~\cite[Lemma~1.2]{Pfp} or the proof of~\cite[Lemma~B.1(c)]{Pmgm}).
\end{proof}

 A finite complex of left $\C$\+comodules $\L^\bu$ is said to have
\emph{projective dimension\/~$\le d$} if one has
$\Hom_{\sD^\b(\C\comodl)}(\L^\bu,\M[n])=0$ for all left $\C$\+comodules
$\M$ and all the integers $n>d$.
 Similarly, a finite complex of left $\D$\+contramodules $\Q^\bu$ is
said to have \emph{injective dimension\/~$\le d$} if one has
$\Hom_{\sD^\b(\D\contra)}(\P,\Q^\bu[n])=0$ for all left
$\D$\+contramodules $\P$ and all $n>d$.

 The bounded above derived category $\sD^-(\D\contra)$ is equivalent
to the homotopy category $\Hot^-(\D\contra_\proj)$ of bounded above
complexes of projective left $\D$\+contramodules.
 Given a complex of right $\D$\+comodules $\N^\bu$ and a bounded above
complex of left $\D$\+contramodules $\P^\bu$, we denote by
$\Ctrtor^\D_*(\N^\bu,\P^\bu)$ the homology vector spaces
$$
 \Ctrtor^\D_n(\N^\bu,\P^\bu)=H^{-n}(\N^\bu\ocn_\D\fF^\bu)
$$
of the contratensor product of the complex $\N^\bu$ with a bounded
above complex of projective left $\D$\+contramodules $\fF^\bu$
quasi-isomorphic to the complex~$\P^\bu$.

 Here the bicomplex $\N^\bu\ocn_\D\fF^\bu$ is presumed to be totalized
by taking infinite direct sums along the diagonals.
 For any complex of right $\D$\+comodules $\N^\bu$, any bounded above
complex of left $\D$\+contramodules $\P^\bu$, and any $k$\+vector
space $V$ there are natural isomorphisms of $k$\+vector spaces
$$
 \Hom_k(\Ctrtor^\D_n(\N^\bu,\P^\bu),V)\simeq
 \Hom_{\sD(\D\contra)}(\P^\bu,\Hom_k(\N^\bu,V)[n]).
$$

 A finite complex of right $\D$\+comodules $\N^\bu$ is said to have
\emph{contraflat dimension\/~$\le d$} if one has
$\Ctrtor_n^\D(\N^\bu,\P)=0$ for all left $\D$\+contramodules $\P$ and
all the integers $n>d$.
 The contraflat dimension of a finite complex of right $\D$\+comodules
$\N^\bu$ is equal to the injective dimension of the finite complex of
left $\D$\+contramodules $\Q^\bu=\Hom_k(\N^\bu,V)$ for any $k$\+vector
space $V\ne0$.

\begin{lem} \label{contraflat-projective-dimension}
 If the coalgebra\/ $\D$ is right cocoherent and left co-Noetherian,
then the contraflat dimension of any finite complex of right\/
$\D$\+comodules\/ $\N^\bu$ does not exceed its projective dimension.
\end{lem}

\begin{proof}
 Let $d$~be the projective dimension of the complex of right
$\D$\+comodules~$\N^\bu$.
 For any finitely copresented right $\D$\+comodule $\L$ there are
natural isomorphisms of complexes of vector spaces
$$
 \Hom_\D(\N^\bu,\L)\simeq\Hom^\D(\L^*,\N^\bu{}^*)
 \simeq (\N^\bu\ocn_\D\L^*)^*
$$
(see Proposition~\ref{finitely-co-presented-hom}), implying natural
isomorphisms of cohomology spaces
$$
 \Hom_{\sD^\b(\D\comodl)}(\N^\bu,\L[n])\simeq
 \Hom_{\sD^\b(\D\contra)}(\L^*,\N^\bu{}^*[n])\simeq
 \Ctrtor^\D_n(\N^\bu,\L^*)^*.
$$
 Since any finitely presented left $\D$\+contramodule $\P$ has
the form $\L^*$ for a certain finitely copresented right
$\D$\+contramodule~$\L$, it follows that the supremum of all
integers~$n$ for which there exists a finitely presented left
$\D$\+contramodule $\P$ with $\Ctrtor^\D_n(\N^\bu,\P)\ne0$ does
not exceed~$d$.

 Furthermore, by Lemma~\ref{tensor-contratensor}(b) the functor of
contratensor product $\ocn_\D$ is isomorphic to the tensor product
functor $\ot_{\D^*}$ over the algebra $\D^*$ on the whole categories
of arbitrary right $\D$\+comodules and left $\D$\+contramodules.
 Besides, the free $\D$\+contramodules are the direct summands of
infinite products of copies of the $\D$\+contramodule~$\D^*$.
 Since the coalgebra $\D$ is left (co-Noetherian and consequently)
cocoherent, the algebra $\D^*$ is right coherent, so infinite
products of flat left $\D^*$\+modules are flat.
 In particular, projective left $\D$\+contramodules are flat as
left modules over~$\D^*$.
 It follows that the functor $\Ctrtor^\D$ is isomorphic to
the derived functor $\Tor^{\D^*}$ of tensor product of (complexes of)
$\D^*$\+modules on the whole domain of definition of the former
derived functor.

 Finally, since the coalgebra $\D$ is left cocoherent, the algebra
$\D^*$ is right coherent and the abelian category of finitely presented
left $\D^*$\+modules is isomorphic to the abelian category of
finitely presented left $\D$\+contramodules.
 Any left $\D^*$\+module is a filtered inductive limit of finitely
presented ones, and the functor of tensor product over $\D^*$
preserves filtered inductive limits.
 The homological dimension of the functor $\Tor^\D_*(\N^\bu,{-})$
on the abelian category of finitely presented left $\D^*$\+modules
does not exceed~$d$, hence the homological dimension of this
derived functor on the abelian category of arbitrary left
$\D^*$\+modules does not exceed~$d$, either.
\end{proof}

 Now we finally come to the main definition of this section.
 Assume that the coalgebra $\C$ is left cocoherent and the coalgebra
$\D$ is right cocoherent.
 A finite complex of $\C$\+$\D$\+bicomodules $\B^\bu$ is called
a \emph{dedualizing complex} for the pair of coalgebras $\C$ and $\D$
if the following conditions hold:
\begin{enumerate}
\renewcommand{\theenumi}{\roman{enumi}}
\item the complex $\B^\bu$ has finite projective dimension as a complex
of left $\C$\+comod\-ules and finite contraflat dimension a complex of
right $\D$\+comodules;
\item the homothety maps
$\C^*\rarrow\Hom_{\sD^\b(\comodr\D)}(\B^\bu,\B^\bu[*])$ and
$\D^*{}^\rop\rarrow\Hom_{\sD^\b(\C\comodl)}(\B^\bu,\B^\bu[*])$
are isomorphisms of graded rings; and \hbadness=1700
\item the bicomodules of cohomology of the complex $\B^\bu$ are
finitely copresented left $\C$\+comodules and finitely copresented
right $\D$\+comodules.
\end{enumerate}

 Here the notation $\comodr\D$ stands for the abelian category of
right $\D$\+comodules, and $\sD^\b(\comodr\D)$ is its bounded derived
category.
 The dedualizing complex $\B^\bu$ itself is viewed as an object
of the bounded derived category $\sD^\b(\C\bicomod\D)$ of
the abelian category $\C\bicomod\D$ of $\C$\+$\D$\+bicomodules.

 The homothety maps are induced by the left action of the algebra
$\C^*$ by right $\D$\+comodule endomorphisms of (every term of)
the complex $\B^\bu$ and the right action of the algebra $\D^*$
by left $\C$\+comodule endomorphisms of~$\B^\bu$.

 We refer to the paper~\cite{Pfp} and the references therein for
a discussion of the classical notion of a dualizing complex over
a pair of noncommutative rings, after which the above definition is
largely modelled.
 A discussion of bicomodules can be found in~\cite[Section~2.6]{Prev}
and the references therein.

\begin{ex} \label{gorenstein-coalgebra-example}
 For any coassociative coalgebra $\C$, the homological dimensions
of the abelian categories of left $\C$\+comodules, right
$\C$\+comodules, and left $\C$\+contramodules coincide
(see~\cite[Section~4.5]{Pkoszul}, cf.~\cite[Corollary~1.9.4]{Pweak}).
 The common value of these three numbers (or infinity) is called
the \emph{homological dimension} of a coalgebra~$\C$.

 Let $\C$ be a left and right cocoherent coalgebra of finite
homological dimension.
 For example, the coalgebra dual to the algebra of quantum formal
power series from Examples~\ref{dual-to-taylor-series} satisfies
these assumptions.
 Then the one-term complex $\B^\bu=\C$ is a dedualizing complex
for the pair of coalgebras $(\C,\C)$, as the conditions~(i\+iii)
are obviously true for~$\B^\bu$.

 More generally, a coalgebra $\C$ is called \emph{left Gorenstein} if
it has finite projective dimension as a left $\C$\+comodule and
finite contraflat dimension as a right $\C$\+comodule.
 (The second condition can be rephrased by saying that the injective
dimension of the left $\C$\+contramodule~$\C^*$ is finite.)
 For any left and right cocoherent, left Gorenstein coalgebra $\C$,
the one-term complex $\B^\bu=\C$ is a dedualizing complex for
the pair of coalgebras $(\C,\C)$.
\end{ex}

\begin{ex} \label{dedualizing-comodule-example}
 Let $R$ be a finitely generated commutative algebra over a field~$k$
and $I\subset R$ be a maximal ideal.
 Then the quotient algebras $R/I^n$ are finite-dimensional, so
their dual vector spaces are cocommutative coalgebras over~$k$, as is
their inductive limit $\C=\varinjlim_n (R/I^n)^*$.
 According to Examples~\ref{dual-to-taylor-series} or
Lemma~\ref{co-contra-Artinian-Noetherian}(b), this coalgebra is
Artinian, and consequently, by Lemma~\ref{co-Noetherian-Artinian}(a),
co-Noetherian and cocoherent.
 The category of $\C$\+comodules is isomorphic to the category of
$I$\+torsion $R$\+modules in the sense of~\cite[Section~1]{Pmgm},
$\,\C\comodl\simeq R\modl_{I\tors}$.
 Moreover, the category of $\C$\+contramodules is isomorphic to
the category of $I$\+contramodule $R$\+modules as defined
in~\cite[Section~2]{Pmgm}, $\,\C\contra\simeq R\modl_{I\ctra}$
(see~\cite[Sections~2.1\+-2.3]{Prev}).

 Of course, the coalgebra $\C$ is cocommutative.
 A complex of $\C$\+comodules $\B^\bu$ is a dedualizing complex for
the pair of coalgebras $(\C,\C)$ in the sense of the above definition
if and only if it is a dedualizing complex of $I$\+torsion $R$\+modules
in the sense of the definition in~\cite[Section~4]{Pmgm}.
 Indeed, the two conditions~(i) are equivalent by
Lemma~\ref{contraflat-projective-dimension}; the two conditions~(ii)
are equivalent because $\fR=\C^*$, and the two conditions~(iii) are
equivalent since, the coalgebra $\C$ being Artinian, a $\C$\+comodule
is finitely copresented if and only if it is Artinian.
 In particular, the dedualizing complex of $I$\+torsion $R$\+modules
constructed in~\cite[Example~4.8]{Pmgm} provides an example of
a dedualizing complex of $\C$\+$\C$\+bicomodules.
\end{ex}

 For any $\C$\+$\D$\+bicomodule $\K$ and any left $\C$\+comodule $\M$,
the $k$\+vector space $\Hom_\C(\K,\M)$ is endowed with the left
$\D$\+contramodule structure of a subcontramodule of
the $\D$\+contramodule $\Hom_k(\K,\M)$.
 Similarly, for any $\C$\+$\D$\+bicomodule $\K$ and any left
$\D$\+contramodule $\P$, the contratensor product $\K\ocn_\D\P$
is endowed with the left $\C$\+comodule structure of a quotient
comodule of the left $\C$\+comodule $\K\ot_k\P$.
 For any $\C$\+$\D$\+bicomodule $\K$, any left $\C$\+comodule $\M$,
and any left $\D$\+contramodule $\P$, there is a natural adjunction
isomorphism of $k$\+vector spaces~\cite[Section~3.1]{Prev}
$$
 \Hom_\C(\K\ocn_\D\P\;\M)\simeq\Hom^\D(\P,\Hom_\C(\K,\M)).
$$

 The following theorem is the main result of this paper.

\begin{thm} \label{dedualizing-bicomodule-duality}
 Given a dedualizing complex\/ $\B^\bu$ for a left cocoherent
coalgebra\/ $\C$ and a right cocoherent coalgebra\/ $\D$ over
a field~$k$, for any symbol\/ $\st=\b$, $+$, $-$, $\varnothing$,
$\abs+$, $\abs-$, or\/~$\abs$ there is an equivalence of derived
categories~\textup{\eqref{mgm-equivalence}}
$$
 \sD^\st(\C\comodl)\simeq\sD^\st(\D\contra)
$$
provided by mutually inverse functors\/ $\boR\Hom_\C(\B^\bu,{-})$
and\/ $\B^\bu\ocn_\D^\boL{-}$.
\end{thm}

\begin{proof}
 (Cf.\ the proofs of~\cite[Theorems~4.9 and~5.10]{Pmgm}.)
 Assume for simplicity of notation that the complex $\B^\bu$ is
concentrated in nonpositive cohomological degrees.
 Let $d$~be an integer greater or equal to both the projective
dimension of the complex $\B^\bu$ viewed as a complex of left
$\C$\+comodules and the contraflat dimension of $\B^\bu$ as
a complex of right $\D$\+comodules.

 To construct the image of a complex of left $\C$\+comodules $\M^\bu$
under the functor $\boR\Hom_\C(\B^\bu,{-})$, one has to choose
an exact sequence of complexes of left $\C$\+comodules
$0\rarrow\M^{\bu}\rarrow\J^{0,\bu}\rarrow\J^{1,\bu}\rarrow\dotsb$
with injective left $\C$\+comodules~$\J^{j,i}$.
 Then one applies the functor $\Hom_\C(\B^\bu,{-})$ to every complex
$0\rarrow\J^{0,i}\rarrow\J^{1,i}\rarrow\J^{2,i}\rarrow\dotsb$,
obtaining a nonnegatively graded complex of left $\D$\+contramodules
$0\rarrow\P^{0,i}\rarrow\P^{1,i}\rarrow\P^{2,i}\rarrow\dotsb$.
 According to the projective dimension condition on the dedualizing
complex~$\B^\bu$, the complex $\P^{\bu,i}$ has zero cohomology
contramodules at the cohomological degrees above~$d$; so it is
quasi-isomorphic to its canonical truncation complex
$\tau_{\le d}\P^{\bu,i}$.
 By the definition, one sets the object $\boR\Hom_\C(\B^\bu,\M^\bu)$
in the derived category $\sD^\st(\D\contra)$ to be represented
by the total complex of the bicomplex $\tau_{\le d}\P^{\bu,\bu}$
concentrated in the cohomological degrees $0\le j\le d$ and $i\in\Z$.

 Similarly, to construct the image of a complex of left
$\D$\+contramodules $\P^\bu$ under the functor $\B^\bu\ocn_\D^\boL{-}$,
one has to choose an exact sequence of complexes of left
$\D$\+contramodules $\dotsb\rarrow\fF^{-1,\bu}\rarrow\fF^{0,\bu}\rarrow
\P^\bu\rarrow0$ with projective left $\D$\+contramodules~$\fF^{j,i}$.
 Then one applies the functor $\B^\bu\ocn_\D{-}$ to every complex
$\dotsb\rarrow\fF^{-2,i}\rarrow\fF^{-1,i}\rarrow\fF^{0,i}\rarrow0$,
obtaining a nonpositively graded complex of left $\C$\+comodules
$\dotsb\rarrow\M^{-2,i}\rarrow\M^{-1,i}\rarrow\M^{0,i}\rarrow0$.
 According to the contraflat dimension condition on the complex
$\B^\bu$, the complex $\M^{\bu,i}$ has zero cohomology comodules at
the cohomological degrees below~$-d$; so it is quasi-isomorphic to
its canonical truncation complex $\tau_{\ge -d}\M^{\bu,i}$.
 One sets the object $\B^\bu\ocn_\D^\boL\P^\bu$ in the derived category
$\sD^\st(\C\comodl)$ to be represented by the total complex of
the bicomplex $\tau_{\ge-d}(\M^{\bu,\bu})$ concentrated in
the cohomological degrees $-d\le j\le0$ and $i\in\Z$.

 These constructions of two derived functors are but particular cases
of the construction of a derived functor of finite homological
dimension spelled out in~\cite[Appendix~B]{Pmgm}.
 According to the results of that appendix, the above constructions
produce well-defined triangulated functors $\boR\Hom_\C(\B^\bu,{-})\:
\sD^\st(\C\comodl)\rarrow\sD^\st(\D\contra)$ and $\B^\bu\ocn_\D^\boL{-}
\:\sD^\st(\D\contra)\rarrow\sD^\st(\C\comodl)$ for any derived category
symbol $\st=\b$, $+$, $-$, $\varnothing$, $\abs+$, $\abs-$, or~$\abs$.
 Moreover, the former functor is right adjoint to the latter one.
 All these assertions only depend on the first condition~(i) in
the definition of a dedualizing complex.

 It remains to prove that the adjunction morphisms are isomorphisms.
 Since the total complexes of finite acyclic complexes of complexes
are absolutely acyclic, in order to check that the morphism
$\P^\bu\rarrow\boR\Hom_\C(\B^\bu\;\B^\bu\ocn_\D^\boL\P^\bu)$ is
an isomorphism in the derived category $\sD^\st(\D\contra)$ for
all the $\st$\+bounded complexes of left $\D$\+contramodules $\P^\bu$
it suffices to consider the case of a one-term complex $\P^\bu=\P$
corresponding to a single $\D$\+contramodule~$\P$.
 Furthermore, since a morphism in $\sD^\b(\D\contra)$ is an isomorphism
whenever it is an isomorphism in $\sD^-(\D\contra)$, one can view
the one-term complex $\P$ as an object of the bounded above derived
category $\sD^-(\D\contra)$ and replace it with a free
$\D$\+contramodule resolution $\fF^\bu$ of the contramodule~$\P$.
 Applying the same totalization argument to the complex~$\fF^\bu$,
the question reduces to proving that the adjunction morphism
$\fF\rarrow\boR\Hom_\C(\B^\bu\;\B^\bu\ocn_\D^\boL\fF)$ is
an isomorphism in $\sD^\b(\D\contra)$ for any free left
$\D$\+contramodule~$\fF$.

 So let $V$ be a $k$\+vector space and $\fF=\Hom_k(\D,V)$ be the free
left $\D$\+contramodule generated by~$V$; then one has
$\B^\bu\ocn_\D^\boL\fF=\B^\bu\ocn_\D\fF=\B^\bu\ot_kV$.
 By the condition~(iii) together with
Lemma~\ref{finitely-cogenerated-cofree-complex}, there exists
a bounded below complex of finitely cogenerated cofree left
$\C$\+comodules $\J^\bu$ together with a quasi-isomorphism of
complexes of left $\C$\+comodules $\B^\bu\rarrow\J^\bu$.
 We have to check that the natural map
$$
 \Hom_k(\D,V)\lrarrow\Hom_\C(\B^\bu\;\J^\bu\ot_kV)
$$
is a quasi-isomorphism of complexes of left $\D$\+contramodules.

 The left-hand side is the projective limit of the vector spaces
$\Hom_k(\E,V)$ over all the finite-dimensional subcoalgebras
$\E\subset\D$, while the right-hand side is the projective limit
of the complexes of vector spaces
$\Hom_\E({}_\E\B^\bu\;{}_\E\J^\bu\ot_kV)$.
 In particular, the map $\D^*\rarrow\Hom_\C(\B^\bu,\J^\bu)$ is
a morphism of complexes of profinite-dimensional topological
vector spaces.
 Being a quasi-isomorphism of complexes of discrete/nontopological
vector spaces (with the topologies forgotten) by the condition~(ii),
it is consequently also a quasi-isomorphism of complexes in
the abelian category of profinite-dimensional topological
vector spaces.

 For any profinite-dimensional topological $k$\+vector space $K$
and any discrete $k$\+vector space $V$ one denotes by
$K\ot\sphat\,V$ the projective limit
$$\textstyle
 K\ot\sphat\,V=\varprojlim_U K/U\ot_kV
$$
taken over all the open subspaces $U\subset K$
\cite[Sections~2.3\+4]{Prev}.
 Equivalently, one can set $W^*\ot\sphat\,V=\Hom_k(W,V)$ for any
discrete $k$\+vector spaces $W$ and~$V$.
 Both the abelian categories of discrete and profinite-dimensional
vector spaces being semisimple, the additive functor
$\ot\sphat\,$ is exact.
 Now one has
$$\textstyle
 \varprojlim_\E\Hom_\E({}_\E\B^\bu\;{}_\E\J^\bu\ot_kV)\simeq
 \varprojlim_\E\Hom_\E({}_\E\B^\bu\;{}_\E\J^\bu)\ot\sphat\,V,
$$
since the complex of $\E$\+comodules ${}_\E\B^\bu$ is finite, while
the terms of the complex ${}_\E\J^\bu$ are finite-dimensional cofree
$\E$\+comodules. 
 Finally, the morphism of complexes in question is obtained by
applying the exact functor ${-}\ot\sphat\,V$ to the quasi-isomorphism
of complexes $\D^*\rarrow\Hom_\C(\B^\bu,\J^\bu)$.

 Similarly, in order to prove that the adjunction morphism
$\B^\bu\ocn_\D^\boL\boR\Hom_\C(\B^\bu,\M^\bu)\allowbreak\rarrow\M^\bu$
is an isomorphism in the derived category $\sD^\st(\C\comodl)$ for
any $\st$\+bounded complex of left $\C$\+comodules $\M^\bu$,
it suffices to check that this morphism is an isomorphism in
$\sD^\b(\C\comodl)$ for any cofree left $\C$\+comodule $\I$
viewed as a one-term complex in $\sD^\b(\C\comodl)$.
 Let $\I=\C\ot_k V$ be a cofree left $\C$\+comodule generated by
a $k$\+vector space $V$; then one has $\boR\Hom_\C(\B^\bu,\I)=
\Hom_\C(\B^\bu,\I)=\Hom_k(\B^\bu,V)$.
 Let $\J^\bu$ be a bounded below complex of finitely cogenerated
cofree right $\D$\+comodules endowed with a quasi-isomorphism of
complexes of right $\D$\+comodules $\B^\bu\rarrow\J^\bu$.
 Then $\Hom_k(\J^\bu,V)$ is a bounded above complex of free left
$\D$\+contramodules quasi-isomorphic to $\Hom_k(\B^\bu,V)$.
 We have to show that the map
$$
 \B^\bu\ocn_\D\Hom_k(\J^\bu,V)\lrarrow\C\ot_kV
$$
induced by the left $\C$\+coaction in $\B^\bu$ is a quasi-isomorphism
(of complexes of left $\C$\+comodules).

 The functors on both sides of our map preserve infinite
direct sums and inductive limits in the argument~$V$, so it suffices
to consider the case $V=k$.
 Passing to the dual vector spaces, we have to check that the map
$\C^*\rarrow\Hom^\D(\J^\bu{}^*\;\B^\bu{}^*)$ is a quasi-isomorphism.
 The latter map is the composition $\C^*\rarrow
\Hom_{\D^\rop}(\B^\bu,\J^\bu)\rarrow\Hom^\D(\J^\bu{}^*\;\B^\bu{}^*)$
of the homothety map of the condition~(ii) and the map induced by
the dualization functor $\N\longmapsto\N^*$.
 It remains to apply Proposition~\ref{finitely-co-presented-hom}(b).
\end{proof}

\Section{MGM Duality for Semialgebras}
\label{dedualizing-bisemimodules}

 Let $\C$ be a coassociative coalgebra over a field~$k$.
 Then the operation of cotensor product $\oc_\C$ (as defined
in the end of Section~\ref{coalgebra-finiteness}) provides the category
of $\C$\+$\C$\+bicomodules $\C\bicomod\C$ with an associative and
unital tensor category structure.
 The $\C$\+$\C$\+bicomodule $\C$ is the unit object.
 A (semiassociative and semiunital) \emph{semialgebra} over $\C$ is
an (associative and unital) algebra object in this tensor category.
 In other words, a semialgebra $\S$ over $\C$ is
a $\C$\+$\C$\+bicomodule endowed with $\C$\+$\C$\+bicomodule
morphisms of \emph{semiunit} $\C\rarrow\S$ and
\emph{semimultiplication} $\S\oc_\C\S\rarrow\S$ satisfying
the conventional associativity and unitality axioms.
 We refer to~\cite[Sections~0.3.1\+-2 and~1.3.1]{Psemi}
and~\cite[Sections~2.5\+-6]{Prev} for further details.

 The cotensor product operation also provides the category of left
$\C$\+comodules $\C\comodl$ with the structure of left module category
over the tensor category $\C\bicomod\C$ and the category of right
$\C$\+comodules $\comodr\C$ with the structure of right module
category over $\C\bicomod\C$.
 Furthermore, the functor of cohomomorphisms $\Cohom_\C$
(see Section~\ref{coalgebra-finiteness}) defined the structure of
a right module category over $\C\bicomod\C$ on the category opposite
to the category of left $\C$\+contramodules $\C\contra^\sop$.
 Given a semialgebra $\S$ over $\C$, one can consider module objects
over the algebra object $\S\in\C\bicomod\C$ in the module categories
$\C\comodl$, \,$\comodr\C$, and $\C\contra^\sop$ over the tensor
category $\C\bicomod\C$.
 This leads to the following definitions.

 A \emph{left semimodule} $\bM$ over $\S$ is a left $\C$\+comodule
endowed with a left $\C$\+comodule morphism of \emph{left
semiaction} $\S\oc_\C\bM\rarrow\bM$ satisfying the associativity and
unitality equations.
 A \emph{right semimodule} $\bN$ over $\S$ is a right $\C$\+comodule
endowed with a right $\C$\+comodule morphism of \emph{right
semiaction} $\bN\oc_\C\S\rarrow\bN$ satisfying the similar equations.
 Finally, a \emph{left semicontramodule} $\bP$ over $\S$ is a left
$\C$\+contramodule endowed with a left $\C$\+contramodule morphism of
\emph{left semicontraaction} $\bP\rarrow\Cohom_\C(\S,\bP)$ satisfying
the dual versions of the same equations.
 The details concerning semimodules can be found in the above
references; and we refer to~\cite[Sections~0.3.4\+-5 and~3.3.1]{Psemi}
and~\cite[Sections~2.5\+-6]{Prev} for further details about
semicontramodules.

 The $k$\+vector space of all morphisms $\bL\rarrow\bM$ in the category
of left $\S$\+semimodules $\S\simodl$ is denoted by $\Hom_\S(\bL,\bM)$.
 Given a left $\C$\+comodule $\L$, the left $\S$\+semi\-module
$\S\oc_\C\L$ is called the left $\S$\+semimodule \emph{induced} from
the left $\C$\+comodule~$\L$.
 For any left $\S$\+semimodule $\bM$, there is a natural isomorphism of
$k$\+vector spaces
$$
 \Hom_\S(\S\oc_\C\L\;\bM)\simeq\Hom_\C(\L,\bM).
$$

 The $k$\+vector space of all morphisms $\bP\rarrow\bQ$ in the category
of left $\S$\+semi\-contramodules $\S\sicntr$ is denoted by
$\Hom^\S(\bP,\bQ)$.
 For any right $\S$\+semimodule $\bN$ and $k$\+vector space $V$,
the left $\C$\+contramodule $\Hom_k(\bN,V)$ has a natural left
$\S$\+semicontramodule structure.
 Given a left $\C$\+contramodule $\Q$, the left $\C$\+contra\-module
$\Cohom_\C(\S,\Q)$ is endowed with a left $\S$\+semicontramodule
structure as a quotient semicontramodule of the left
$\S$\+semicontramodule $\Hom_k(\S,\Q)$.
 The left $\S$\+semicontramodule $\Cohom_\C(\S,\Q)$ is called
the left $\S$\+semicontramodule \emph{coinduced} from
the left $\C$\+contramodule~$\Q$.
 For any left $\S$\+semicontramodule $\bP$, there is a natural isomorphism
of $k$\+vector spaces
$$
 \Hom^\S(\bP,\Cohom_\C(\S,\Q))\simeq\Hom^\C(\bP,\Q).
$$

 Our next aim is to define the operation of \emph{contratensor product}
$\bN\Ocn_\S\bP$ of a right $\S$\+semimodule $\bN$ and a left
$\S$\+semicontramodule~$\bP$ \cite[Sections~0.3.7 and~6.1.1\+-2]{Psemi}.
 The idea is that $\bN\Ocn_\S\bP$ is a $k$\+vector space for which
the natural isomorphism
\begin{equation} \label{contratensor-semicontramodule-hom}
 \Hom_k(\bN\Ocn_\S\bP,\>V)\simeq\Hom^\S(\bP,\Hom_k(\bN,V))
\end{equation}
holds for any $k$\+vector space~$V$.
 This condition determines the $k$\+vector space $\bN\Ocn_\S\bP$
uniquely up to a natural isomorphism.
 The following explicit construction shows that such a vector space
exists.

 The contratensor product $\bN\Ocn_\S\bP$ is the cokernel of
(the difference of) the pair of natural $k$\+linear maps
$$
 (\bN\oc_\C\S)\ocn_\C\bP\birarrow\bN\ocn_\C\bP.
$$
 Here the first map is induced by the right $\S$\+semiaction morphism
$\bN\oc_\C\S\rarrow\bN$, while the second map is the composition of
the left $\S$\+semicontraaction morphism $\bP\rarrow\Cohom_\C(\S,\bP)$
and the natural ``evaluation'' map
$$
 \eta_\S\:(\bN\oc_\C\S)\ocn_\C\Cohom_\C(\S,\bP)\lrarrow\bN\ocn_\C\bP.
$$

 The ``evaluation'' map is defined for any two coalgebras $\C$ and $\D$
over~$k$, a $\C$\+$\D$\+bicomodule $\K$, a right $\C$\+comodule $\N$,
and a left $\C$\+contramodule $\P$,
$$
 \eta_\K\:(\N\oc_\C\K)\ocn_\D\Cohom_\C(\K,\P)\lrarrow\N\ocn_\C\P,
$$
and can be characterized by the condition that the dual map
$\eta_K^*=\Hom_k(\eta_\K,k)$ is equal to the map
$$
 \Hom^\C(\P,\N^*)\lrarrow\Hom^\D(\Cohom_\C(\K,\P),\Cohom_\C(\K,\N^*))
$$
provided by the functor $\Cohom_\C(\K,{-})\:\C\contra\rarrow\D\contra$.
 Even more explicitly, the $k$\+linear map~$\eta_\K$ is constructed as
the unique map forming a commutative square with the composition
of maps
$$
 (\N\oc_\C\K)\ot_k\Hom_k(\K,\P)\lrarrow\N\ot_k\K\ot_k\Hom_k(\K,\P)
 \lrarrow\N\ot_k\P
$$
and the natural surjections.
 We refer to~\cite[Section~6.1.1]{Psemi} for further details.

 For any right $\C$\+comodule $\N$ and left $\S$\+semicontramodule
$\bP$, there is a natural isomorphism of $k$\+vector
spaces~\cite[Section~6.1.2]{Psemi}
$$
 (\N\oc_\C\S)\Ocn_\S\bP\simeq\N\ocn_\C\bP.
$$

 The category $\S\simodl$ of left $\S$\+semimodules is abelian provided
that $\S$ is an injective right $\C$\+comodule.
 In fact, $\S$ is an injective right $\C$\+comodule \emph{if and
only if} the category $\S\simodl$ is abelian \emph{and} the forgetful
functor $\S\simodl\rarrow\C\comodl$ is exact.
 Similarly, the category $\S\sicntr$ of left $\S$\+semicontramodules is
abelian provided that $\S$ is an injective left $\C$\+comodule.
 In fact, $\S$ is an injective left $\C$\+comodule \emph{if and
only if} the category $\S\sicntr$ is abelian \emph{and} the forgetful
functor $\S\sicntr\rarrow\C\contra$ is exact.
 (See~\cite[Proposition~2.5]{Prev} for a proof of the dual versions
of these results.)

\bigskip

 Let $\S$ be a semialgebra over a coalgebra $\C$ over~$k$ and $\T$ be
a semialgebra over a coalgebra $\D$ over~$k$.
 An \emph{$\S$\+$\T$\+bisemimodule} $\bK$ is a $\C$\+$\D$\+bicomodule
endowed with a left $\S$\+semimodule and a right $\T$\+semimodule
structures such that the semiaction maps $\S\oc_\C\bK\rarrow\bK$ and
$\bK\oc_\D\T\rarrow\bK$ are morphisms of $\C$\+$\D$\+bicomodules
which commute with each other in the sense that the two compositions
$\S\oc_\C\bK\oc_\D\T\rarrow\bK\oc_\D\T\rarrow\bK$ and
$\S\oc_\C\bK\oc_\D\T\rarrow\S\oc_\C\bK\rarrow\bK$ coincide.
 Alternatively, one can define an $\S$\+$\T$\+bisemimodule as
a $\C$\+$\D$\+bicomodule endowed with a \emph{bisemiaction} map
$\S\oc_\C\bK\oc_\D\T\rarrow\bK$, which must be a morphism of
$\C$\+$\D$\+bicomodules satisfying the associativity and
unitality axioms.

 Let $\bK$ be an $\S$\+$\T$\+bisemimodule and $\bP$ be a left
$\T$\+semicontramodule.
 Assuming that $\S$ is an injective right $\C$\+comodule,
the contratensor product $\bK\Ocn_\T\bP$ then has a natural left
$\S$\+semimodule structure.
 Similarly, let $\bK$ be an $\S$\+$\T$\+bisemimodule and $\bM$ be
a left $\S$\+semimodule.
 Assuming that $\T$ is an injective left $\D$\+comodule,
the $k$\+vector space of left $\S$\+semimodule morphisms
$\Hom_\S(\bK,\bM)$ has a natural left $\T$\+semicontramodule
structure~\cite[Section~6.1.3]{Psemi}.

 Whenever $\S$ is an injective right $\C$\+comodule and $\T$ is
an injective left $\D$\+comodule, for any $\S$\+$\T$\+bisemimodule
$\bK$, any left $\S$\+semimodule $\bM$, and any right
$\T$\+semicontra\-module $\bP$, there is a natural adjunction
isomorphism of $k$\+vector spaces~\cite[Section~6.1.4]{Psemi}
$$
 \Hom_\S(\bK\Ocn_\T\bP,\>\bM)\simeq\Hom^\T(\bP,\Hom_\S(\bK,\bM)).
$$

 There are also some other situations in which there is a natural
left $\S$\+semimodule structure on the contratensor product
$\bK\Ocn_\T\bP$ and a natural left $\T$\+semicontramodule structure
on the space of homomorphisms $\Hom_\S(\bK,\bM)$.
 The case of $\S=\bK=\T$ is of particular interest.
 For any semialgebra $\S$ over a coalgebra $\C$, the functors
$\Phi_\S\:\bP\longmapsto\S\Ocn_\S\bP$ and $\Psi_\S\:
\bM\longmapsto\Hom_\S(\S,\bM)$ establish an equivalence between
the exact categories of $\C$\+injective left $\S$\+semimodules and
$\C$\+projective left $\S$\+semicontramodules~\cite[Section~6.2]{Psemi}.
 This equivalence forms a commutative square with the forgetful
functors and the equivalence between the additive categories of
injective left $\C$\+comodules and projective left $\C$\+contramodules
$\C\comodl_\inj\simeq\C\contra_\proj$ provided by the functors
$\Phi_\C=\C\ocn_\C{-}$ and $\Psi_\C=\Hom_\C(\C,{-})$.

\bigskip

 Let $\S$ be a semialgebra over a coalgebra $\C$ over~$k$ and $\T$ be
a semialgebra over a coalgebra $\D$ over~$k$.
 Assume that $\S$ is an injective right $\C$\+comodule and $\T$ is
an injective left $\D$\+comodule.
 So the categories $\S\simodl$ and $\T\sicntr$ are abelian.

\begin{prop}
\textup{(a)} There are enough injective objects in the abelian
category\/ $\S\simodl$.
 A left\/ $\S$\+semimodule is injective if and only if it is a direct
summand of a left\/ $\S$\+semimodule of the form\/
$\Phi_\S(\Hom_k(\S,V))$, where $V$ is a $k$\+vector space.
 Furthermore, the forgetful functor\/ $\S\simodl\rarrow\C\comodl$
preserves injectives. \par
\textup{(b)} There are enough projective objects in the abelian
category\/ $\T\sicntr$.
 A left\/ $\T$\+semicontramodule is projective if and only if it is
a direct summand of a left\/ $\T$\+semicontramodule of the form\/
$\Psi_\T(\T\ot_kV)$, where $V$ is a $k$\+vector space.
 Furthermore, the forgetful functor\/ $\T\sicntr\rarrow\D\contra$
preserves projectives.
\end{prop}

\begin{proof}
 A proof of this result under slightly more restrictive assumptions
(of a semialgebra injective over its coalgebra on both sides) can be
found in~\cite[Proposition~3.5]{Prev}.
 In the general case, there is an argument based on the results
of~\cite[Section~6.2]{Psemi}, proceeding as follows.

 Part~(a): the forgetful functor $\S\simodl\rarrow\C\comodl$ preserves
injectives, since it has an exact left adjoint functor $\S\oc_\C{-}$
assigning to a left $\C$\+comodule the induced left $\S$\+semimodule.
 To prove that the left $\S$\+semimodule
$\bM=\S\Ocn_\S\Hom_k(\S,V)$ is injective, one first notices that
$\Hom_k(\S,V)$ is a projective left $\C$\+contramodule, hence
$\bM$ is an injective left $\C$\+comodule and
$\Hom_\S(\S,\bM)=\Psi_\S(\bM)\simeq\Hom_k(\S,V)$.
 Applying~\cite[Proposition~6.2.2(a)]{Psemi} for $\T=\bK=\S$,
one computes that
$$
 \Hom_\S(\bL,\bM)\simeq\Hom_k(\bL,V)
$$
for any left $\S$\+semimodule~$\bL$.
 This proves that $\bM$ is injective; and in order to show that $\bL$
can be embedded into a left $\S$\+semimodule of the form
$\Phi_\S(\Hom_k(\S,V))$, it suffices to take $V=\bL$.

 Part~(b): the forgetful functor $\T\sicntr\rarrow\D\contra$ preserves
projectives, since it has an exact right adjoint functor
$\Cohom_\D(\T,{-})$ assigning to a left $\D$\+contramodule
the coinduced left $\T$\+semicontramodule.
 To prove that the left $\T$\+semicontramodule
$\bP=\Hom_\T(\T,\>\T\ot_kV)$ is projective, one first notices that
$\T\ot_kV$ is an injective left $\D$\+comodule, hence $\bP$ is
a projective left $\D$\+contramodule and
$\T\Ocn_\T\bP=\Phi_\T(\bP)\simeq\T\ot_kV$.
 Applying~\cite[Proposition~6.2.3(a)]{Psemi} for $\S=\bK=\T$,
one computes that
$$
 \Hom^\T(\bP,\bQ)\simeq\Hom_k(V,\bQ)
$$
for any left $\T$\+semicontramodule~$\bQ$.
 This proves that $\bP$ is projective; and in order to show that
$\bQ$ is a quotient object of a left $\T$\+semicontramodule of
the form $\Psi_\T(\T\ot_kV)$, it suffices to take $V=\bQ$.
\end{proof}

 A finite complex of left $\S$\+semimodules $\bL^\bu$ is said to have
\emph{projective dimension\/~$\le d$} if one has
$\Hom_{\sD^\b(\S\simodl)}(\bL^\bu,\bM[n])=0$ for all left
$\S$\+semimodules $\bM$ and all the integers $n>d$.
 The projective dimension of the complex of left $\S$\+semimodules
$\bL^\bu=\S\oc_\C\L^\bu$ induced from a finite complex of left
$\C$\+comodules $\L^\bu$ does not exceed the projective dimension
of the complex of left $\C$\+comodules~$\L^\bu$.

 Similarly, a finite complex of left $\T$\+semicontramodules $\bQ^\bu$
is said to have \emph{injective dimension\/~$\le d$} if one has
$\Hom_{\sD^\b(\T\sicntr)}(\bP,\bQ^\bu[n])=0$ for all left
$\T$\+semicontramodules $\bP$ and all $n>d$.
 The injective dimension of the complex of left $\T$\+semicontramodules
$\bQ^\bu=\Cohom_\D(\T,\Q^\bu)$ induced from a finite complex of
left $\D$\+contramodules $\Q^\bu$ does not exceed the injective
dimension of the complex of left $\D$\+contramodules~$\Q^\bu$.

 The bounded above derived category $\sD^-(\T\sicntr)$ is equivalent
to the homotopy category $\Hot^-(\T\sicntr_\proj)$ of bounded above
complexes of projective left $\T$\+semicontramodules.
 Given a complex of right $\T$\+semimodules $\bN^\bu$ and
a bounded above complex of left $\T$\+semicontramodules $\bP^\bu$,
we denote by $\CtrTor^\T_*(\bN^\bu,\bP^\bu)$ the homology
vector spaces
$$
 \CtrTor^\T_n(\bN^\bu,\bP^\bu)=H^{-n}(\bN^\bu\Ocn_\T\bF^\bu)
$$
of the contratensor product of the complex of right $\T$\+semimodules
$\bN^\bu$ with a bounded above complex of projective left
$\T$\+semicontramodules $\bF^\bu$ quasi-isomorphic to~$\bP^\bu$.

 Here one totalizes the bicomplex $\bN^\bu\Ocn_\T\bF^\bu$ by taking
infinite direct sums along the diagonals.
 For any complex of right $\T$\+semimodules $\bN^\bu$, any bounded
above complex of left $\T$\+semicontramodules $\bP^\bu$, and
a $k$\+vector space $V$ there are natural isomorphisms of
$k$\+vector spaces
$$
 \Hom_k(\CtrTor^\T_n(\bN^\bu,\bP^\bu),V)\simeq
 \Hom_{\sD(\T\sicntr)}(\bP^\bu,\Hom_k(\bN^\bu,V)[n]).
$$
 Alternatively, the derived functor $\CtrTor^\T$ can be computed using
$\T/\D$\+projective (or $\T/\D$\+contraflat) resolutions of the first
argument and $\D$\+projective resolutions of the second
argument~\cite[Sections~6.4\+-5]{Psemi}.

 A finite complex of right $\T$\+semimodules $\bN^\bu$ is said to have
\emph{contraflat dimension\/~$\le\nobreak d$} if one has
$\CtrTor_n^\T(\bN^\bu,\bP)=0$ for all left $\T$\+semicontramodules
$\bP$ and all the integers $n>d$.
 The contraflat dimension of the complex of right $\T$\+semimodules
$\bN^\bu=\N^\bu\oc_\D\T$ induced from a finite complex of right
$\D$\+comodules $\N^\bu$ does not exceed the contraflat dimension
of the complex of right $\D$\+comodules~$\N^\bu$.
 The contraflat dimension of a finite complex of
right $\T$\+semimodules $\bN^\bu$ is equal to the injective dimension
of the finite complex of left $\T$\+semicontramodules
$\bQ^\bu=\Hom_k(\bN^\bu,V)$ for any $k$\+vector space $V\ne0$.

\bigskip

 Now we come to the main definition of this section.
 Let $\S$ be a semialgebra over a coalgebra $\C$ over~$k$ and $\T$ be
a semialgebra over a coalgebra~$\D$ over~$k$.
 Assume that $\S$ is an injective right $\C$\+comodule, $\T$ is
an injective left $\D$\+comodule, the coalgebra $\C$ is
left cocoherent, and the coalgebra $\D$ is right cocoherent.
 A \emph{dedualizing complex} for $\S$ and $\T$ is defined as a triple
consisting of a finite complex of $\S$\+$\T$\+bisemimodules $\bB^\bu$,
a finite complex of $\C$\+$\D$\+bicomodules $\B^\bu$, and a morphism
of complexes of $\C$\+$\D$\+bicomodules $\B^\bu\rarrow\bB^\bu$
with the following properties:
\begin{enumerate}
\renewcommand{\theenumi}{\roman{enumi}}
\setcounter{enumi}{3}
\item $\B^\bu$ is a dedualizing complex for the pair of coalgebras
$\C$ and $\D$, that is the conditions~(i\+iii) of
Section~\ref{dedualizing-bicomodules} are satisfied;
\item the morphism of complexes of left $\S$\+semimodules
$\S\oc_\C\B^\bu\rarrow\bB^\bu$ induced by the morphism of
complexes of left $\C$\+comodules $\B^\bu\rarrow\bB^\bu$ is
a quasi-isomorphism;
\item the morphism of complexes of right $\T$\+semimodules
$\B^\bu\oc_\D\T\rarrow\bB^\bu$ induced by the morphism of
complexes of right $\D$\+comodules $\B^\bu\rarrow\bB^\bu$ is
a quasi-isomorphism.
\end{enumerate}

 It follows from the conditions~(i) and~(v) that the complex $\bB^\bu$
has finite projective dimension as a complex of left $\S$\+semimodules.
 Similarly, it follows from the conditions~(i) and~(vi) that
the complex $\bB^\bu$ has finite contraflat dimension as a complex
of right $\T$\+semimodules.

 Abusing the terminology, we will sometimes say that the complex
of $\S$\+$\T$\+bisemi\-modules $\bB^\bu$ is a dedualizing complex
(for the semialgebras $\S$ and~$\T$).

\begin{exs}
 Let $\S$ be a semialgebra over a coalgebra $\C$ over~$k$ such that
the coalgebra $\C$ is left and right cocoherent and left Gorenstein
(see Example~\ref{gorenstein-coalgebra-example}), while
the semialgebra $\S$ is an injective left $\C$\+comodule and
an injective right $\C$\+comodule.
 Then the triple consisting of the $\S$\+$\S$\+bisemimodule $\S$
(viewed as a one-term complex of $\S$\+$\S$\+bisemimodules),
the $\C$\+$\C$\+bicomodule $\C$ (also viewed as a one-term
complex of $\C$\+$\C$\+bicomodules), and the semiunit morphism
$\C\rarrow\S$ is a dedualizing complex for the pair of
semialgebras~$(\S,\S)$.

 In particular, any semialgebra $\S$ over a cosemisimple coalgebra $\C$
satisfies the above conditions (as cosemisimple coalgebras are
co-Noetherian and of homological dimension~$0$), so $\S$ is
a dedualizing complex of $\S$\+$\S$\+bisemimodules.
 This situation is a rather trivial case for the following theorem, 
though, as in this case the \emph{abelian} categories $\S\simodl$
and $\S\sicntr$ are already equivalent (the underived functors
$\Psi_\S=\Hom_\S(\S,{-})$ and $\Phi_\S=\S\Ocn_\S{-}$ providing
the equivalence).

 For example, let $G$ be a locally compact totally disconnected
(locally profinite) group and $H\subset G$ be a compact open subgroup.
 Let $k$ be a field.
 Then the $k$\+vector space $\C=k(H)$ of locally constant $k$\+valued
functions on $H$ has a natural structure of coalgebra over~$k$.
 Moreover, the $k$\+vector space $\S=k(G)$ of compactly supported
locally constant $k$\+valued functions on $G$ has a natural structure
of semialgebra over~$\C$.
 The semialgebra $\S$ is always an injective left and right
$\C$\+comodule.
 The category of (left or right) $\S$\+semimodules is isomorphic to
the abelian category $G\smooth_k$ of smooth $G$\+modules over~$k$,
while the category of (left or right) $\S$\+semicontramodules is
isomorphic to the abelian category $G\contra_k$ of $G$\+contramodules
over~$k$ (see the introduction to~\cite{Psm} and
the references therein).

 Assume that the proorder of the profinite group $H$ is not divisible
by the characteristic of the field~$k$.
 In particular, $G$ can be an arbitrary locally profinite group and $k$
a field of characteristic~$0$, or $G$ can be a $p$\+adic Lie group,
$H\subset G$ an open pro-$p$-subgroup, and $k$ a field
of characteristic different from~$p$.
 Then the coalgebra $\C=k(H)$ is cosemisimple.
 So the semialgebra $\S=k(G)$ is a dedualizing complex of
bisemimodules over itself.
 Moreover, the abelian categories $G\smooth_k$ and $G\contra_k$
of left $\S$\+semimodules and left $\S$\+semicontramodules
are equivalent.

 The situation in the natural characteristic~$p$ is more interesting.
 Let $G$ be a $p$\+adic Lie group and $H\subset G$ be a compact open
subgroup such that $H$ has no elements of order~$p$.
 Let $k$ be a field of characteristic~$p$.
 Then the coalgebra $\C=k(H)$ is left and right Artinian, since its
dual algebra $\C^*=k[[H]]$ is left and right Noetherian.
 Furthermore, the coalgebra $\C$ has finite homological dimension
(equal to the dimension of the group~$G$).
 Thus the semialgebra $\S=k(G)$ is a dedualizing complex of
$\S$\+$\S$\+bisemimodules.
 Hence the following theorem applies (cf.~\cite{Psm}).
\end{exs}

\begin{thm}
 Let\/ $\S$ be a semialgebra over a coalgebra\/ $\C$ and\/ $\T$ be
a semialgebra over a coalgebra\/ $\D$ over a field~$k$.
 Assume that the coalgebra\/ $\C$ is left cocoherent, the coalgebra\/
$\D$ is right cocoherent, the semialgebra\/ $\S$ is an injective
right\/ $\C$\+comodule, and the semialgebra\/ $\T$ is an injective
left\/ $\D$\+comodule.
 Let\/ $\B^\bu\rarrow\bB^\bu$ be a dedualizing complex for
the semialgebras\/ $\S$ and\/~$\T$.
 Then for any symbol\/ $\st=\b$, $+$, $-$, $\varnothing$, $\abs+$,
$\abs-$, or\/~$\abs$ there is an equivalence of triangulated
categories~\textup{\eqref{semialg-mgm-equivalence}}
$$
 \sD^\st(\S\simodl)\simeq\sD^\st(\T\sicntr)
$$
provided by mutually inverse functors\/ $\boR\Hom_\S(\bB^\bu,{-})$
and\/ $\bB^\bu\Ocn_\T^\boL{-}$.
\end{thm}

\begin{proof}
 The constructions of the derived functors $\boR\Hom_\S(\bB^\bu,{-})$
and $\bB^\bu\Ocn_\T^\boL{-}$ proceed exactly in the same way as in
the proof of Theorem~\ref{dedualizing-bicomodule-duality} (with $\C$
replaced by $\S$, left $\C$\+comodules by left $\S$\+semimodules, $\D$
replaced by $\T$, left $\D$\+contramodules by left
$\T$\+semicontramodules, and the complex of bicomodules $\B^\bu$
replaced by the complex of bisemimodules~$\bB^\bu$).
 The only property of a finite complex of $\S$\+$\T$\+bisemimodules
$\bB^\bu$ that is used in these constructions is the finiteness of
the projective and contraflat dimensions.
 The result of~\cite[Appendix~B]{Pmgm} tells that the two derived
functors so obtained are adjoint to each other.

 Next it is noticed that the two pairs of adjoint derived functors
corresponding to the complex of $\C$\+$\D$\+bicomodules $\B^\bu$
and the complex of $\S$\+$\T$\+bisemimodules $\bB^\bu$ form
commutative diagrams with the forgetful functors (cf.\ the proof
of~\cite[Theorem~5.6]{Pfp})
$$\dgARROWLENGTH=4em
\!\!\!
\begin{diagram} 
\node{\sD^\st(\S\simodl)}
\arrow[2]{e,t}{\boR\Hom_\S(\bB^\bu,{-})}
\arrow{s}
\node[2]{\sD^\st(\T\sicntr)} \arrow{s} \\
\node{\sD^\st(\C\comodl)}
\arrow[2]{e,t}{\boR\Hom_\C(\B^\bu,{-})}
\node[2]{\sD^\st(\D\contra)}
\end{diagram}
\,
\begin{diagram} 
\node{\sD^\st(\S\simodl)}
\arrow{s}
\node[2]{\sD^\st(\T\sicntr)} \arrow{s}
\arrow[2]{w,t}{\bB^\bu\Ocn_\T^\boL{-}} \\
\node{\sD^\st(\C\comodl)}
\node[2]{\sD^\st(\D\contra)}
\arrow[2]{w,t}{\B^\bu\ocn_\D^\boL{-}}
\end{diagram}
\!\!\!
$$
 This follows from the conditions~(v\+vi); the argument is that
the total complex of a finite acyclic complex of complexes is
absolutely acyclic.
 The observation that the functor of contratensor product
${-}\Ocn_\T\bF$ with a projective left $\T$\+semicontramodule $\bF$
is exact on the abelian category of right $\T$\+semimodules
$\simodr\T$ plays a role here (see the natural
isomorphism~\eqref{contratensor-semicontramodule-hom}).

 Finally, just as in the proof of
Theorem~\ref{dedualizing-bicomodule-duality}, checking that
the adjunction morphisms for the derived functors
$\boR\Hom_\S(\bB^\bu,{-})$ and $\bB^\bu\Ocn_\T^\boL{-}$
are isomorphisms reduces to the case of the bounded derived
categories, $\st=\b$.
 This is a conventional derived category; and for all the conventional
derived categories ($\st=\b$, $+$, $-$, or~$\varnothing$)
the forgetful functors $\sD^\st(\S\simodl)\rarrow
\sD^\st(\C\comodl)$ and $\sD^\st(\T\sicntr)\rarrow
\sD^\st(\D\contra)$ are conservative.
 So it suffices to show that the images of the adjunction morphisms
under the forgetful functors are isomorphisms.
 Similarly to the proof of~\cite[Theorem~5.6]{Pfp}, one observes
that these images are nothing but the adjunction morphisms for
the derived functors $\boR\Hom_\C(\B^\bu,{-})$ and
$\B^\bu\ocn_\D^\boL{-}$.
 According to the result of
Theorem~\ref{dedualizing-bicomodule-duality}, we already know
that the latter are isomorphisms.
\end{proof}

\end{document}